\documentclass[11pt]{article}
\usepackage{geometry}
\usepackage{subcaption}
\usepackage{amsthm,amsmath,amssymb,amsfonts,hyperref}
\usepackage{epsfig,color}
\usepackage{enumerate}

\usepackage{arydshln}	
\usepackage{pdflscape} 
\usepackage{multirow}
\usepackage{rotating} 
\usepackage{ulem} 

\newtheorem{thm}{Theorem}[section]
\newtheorem{defn}[thm]{Definition}
\newtheorem{prop}[thm]{Proposition}
\newtheorem{cor}[thm]{Corollary}
\newtheorem{rmk}[thm]{Remark}
\newtheorem{lma}[thm]{Lemma}

\newtheorem{exm}[thm]{Example}

\newcommand{\eq}{\begin{equation}}
\newcommand{\qe}{\end{equation}}

\def\N{{\rm I\kern-0.16em N}}
\def\R{{\rm I\kern-0.16em R}}

\def\P{{\rm I\kern-0.16em P}}
\def\F{{\rm I\kern-0.16em F}}
\def\B{{\rm I\kern-0.16em B}}
\def\C{{\rm I\kern-0.46em C}}
\def\G{{\rm I\kern-0.50em G}}
\def\Z{{\mathbb{Z}}}

\numberwithin{equation}{section}

\begin{document}

\normalem 

\title{First order covariance inequalities via Stein's method}

\date{}
\renewcommand{\thefootnote}{\fnsymbol{footnote}}
\begin{small}
  \author{Marie Ernst\footnotemark[1], \, 
Gesine Reinert\footnotemark[2] \, and Yvik Swan\footnotemark[1]}
\end{small}
\footnotetext[1]{Universit\'e de Li\`ege, corresponding author Yvik
  Swan: yswan@uliege.be.}
\footnotetext[2]{University of Oxford.}

\maketitle

\begin{abstract} We propose probabilistic representations for inverse
  Stein operators (i.e.\ solutions to Stein equations) under general
  conditions; in particular we deduce new simple expressions for the
  Stein kernel. These representations allow to deduce uniform and
  non-uniform Stein factors (i.e.\ bounds on solutions to Stein
  equations) and lead to new covariance identities expressing the
  covariance between arbitrary functionals of an arbitrary {univariate} target in
  terms of a weighted covariance of the derivatives of the
  functionals. Our weights are explicit, easily computable in most
  cases, and expressed in terms of objects familiar within the context
  of Stein's method.  Applications of the Cauchy-Schwarz inequality to
  these weighted covariance identities lead to sharp upper and lower
  covariance bounds and, in particular, weighted Poincar\'e
  inequalities. Many examples are given and, in particular, classical
  variance bounds due to Klaassen, Brascamp and Lieb or Otto and Menz
  are corollaries. Connections with more recent literature are also
  detailed.
 
\end{abstract}


\section{Introduction}

Much attention has been given in the literature to the problem of
providing sharp tractable estimates on the variance of functions of
random variables. Such estimates are directly related to fundamental
considerations of pure mathematics (e.g.,\ isoperimetric, logarithmic
Sobolev and Poincar\'e inequalities), as well as essential issues from
statistics (e.g.,\ Cramer-Rao bounds, efficiency and asymptotic
relative efficiency computations, maximum correlation coefficients,
and concentration inequalities).

One of the starting points of this line of research is Chernoff's
famous result from \cite{C80} which states that, if
$N \sim \mathcal{N}(0, 1)$, then
\begin{equation}
  \label{eq:99} 
\mathbb{E}[g'(N)]^2 \le \mathrm{Var}[g(N)] \le \mathbb{E}[g'(N)^2]
\end{equation}
for all sufficiently regular functions $g:\R\to\R$.  Chernoff obtained
the upper bound by exploiting orthogonality properties of the family
of Hermite polynomials.  The upper bound in \eqref{eq:99} is, in fact,
already available in \cite{Na58} and is also a special case of the
central inequality in \cite{BrLi76}, see below. 
Cacoullos \cite{C82} extends Chernoff's bound to a wide class of
univariate distributions (including discrete distributions) by proving
that if $X \sim p$ has a density function $p$ with respect to the
Lebesgue measure then
\begin{align}
  \label{eq:69}
  \frac{\mathbb{E}[\tau_p(X)\, g'(X)]^2}{\mathrm{Var}[X]} \le
  \mathrm{Var}[g(X)] \le  \mathbb{E} \left[  \tau_p(X) \, g'(X )^2 \right] 
\end{align}
with
$ \tau_p(x) = {p(x)}^{-1} \int_x^{\infty} (t- \mathbb{E}[X]) p(t)
dt$. It is easy to see that, if $p$ is the standard normal density,
then $\tau_p(x) = 1$ so that \eqref{eq:69} contain \eqref{eq:99}.
Cacoullos also obtains a similar bound as \eqref{eq:69} for discrete
distributions on the positive integers, where the derivative is
replaced by the forward difference and the weight becomes
$ \tau_p(x) = {p(x)}^{-1} \sum_{t=x+1}^{\infty} t p(t).
$ 

Variance inequalities such as \eqref{eq:69} are closely related to the
celebrated \emph{Brascamp-Lieb inequality} from \cite{BrLi76} which,
in dimension 1, states that if $X \sim p$ and $p$ is strictly
log-concave then
\begin{equation}
  \label{eq:brascamplieb}
  \mathrm{Var}[g(X)] \le \mathbb{E}\left[\frac{(g'(X))^2}{(-\log p)''(X)}\right]
\end{equation}
for all sufficiently regular functions $g$.  In fact, the upper bound
from \eqref{eq:99} is an immediate consequence of
\eqref{eq:brascamplieb} because, if $p$ is the standard Gaussian
density, then $(-\log p)''(x) \equiv 1$.  The Brascamp-Lieb inequality
is proved in \cite{menz2013uniform} to be a consequence of Hoeffding's
classical covariance inequality from
\cite{hoffding1940masstabinvariante}, which states that if $(X, Y)$ is a
continuous bivariate random vector with cumulative distribution
$H(x, y)$ and marginal cdfs $F(x), G(x)$ then
\begin{equation}
  \label{eq:Hoefdingcov}
\mathrm{Cov}[f(X), g(Y)] = \int_{-\infty}^{\infty}\int_{-\infty}^{\infty}
f'(x)\Big(H(x, y) - F(x)G(y)\Big)g'(y) \, \mathrm{d}x \, \mathrm{d}y
\end{equation}
under weak assumptions on $f, g$ (see e.g.\
\cite{cuadras2002covariance}). 
The freedom of choice in the test
functions $f, g$ in \eqref{eq:Hoefdingcov} is exploited by
\cite{menz2013uniform} to prove that, if $X$ has a $C^2$ strictly
convex absolutely continuous density $p$ then the 
\emph{asymmetric Brascamp-Lieb inequality} holds:
\begin{equation}
  \label{eq:ottomenz}
  \left| \mathrm{Cov}[f(X), g(X)]  \right| \le \sup_x \left\{
    \frac{|f'(x)|}{ (\log p)''(x)} \right\} \mathbb{E}\big[\left| g'(X)
  \right|\big]. 
\end{equation} 
Identity \eqref{eq:Hoefdingcov} and inequalities
\eqref{eq:brascamplieb} and \eqref{eq:ottomenz} are extended to the
multivariate setting in \cite{carlen2013asymmetric} {which also gives}
 connections with
logarithmic Sobolev inequalities for spin systems and related
inequalities for log-concave densities.
{This material is revisited and extended in} \cite{saumard2018efron, saumard2018weighted,
  saumard2017isoperimetric},
providing applications in the context of
isoperimetric inequalities and weighted Poincar\'e inequalities. 
{In} \cite{cuadras2002covariance}
 the identity
\eqref{eq:Hoefdingcov} is proved in all generality and used to provide
 expansions for the covariance in terms of canonical
correlations and variables.

Further generalizations of Chernoff's bounds are provided in
\cite{chen1985poincare,CaPa85, CP86}, and \cite{karlin1993general}
(e.g., Karlin \cite{karlin1993general} deals with the entire class of
log-concave distributions).  See also
\cite{borovkov1984inequality,CP89, korwar1991characterizations,
  papathanasiou1995characterization, CP95} for the connection with
probabilistic characterizations and other properties. 
Similar inequalities were 
obtained -- often by exploiting
properties of suitable families of orthogonal polynomials -- for
univariate functionals of some specific multivariate distributions
e.g.,\ in
\cite{cacoullos1992lower,cacoullos1998short,chang1999variance,landsman2013note,
  afendras2014note,landsman2015some}. {A historical overview as
  well as a description of the connection between such bounds, the
  so-called Stein identities from Stein's method (see below) and
  Sturm-Liouville theory (see Section \ref{sec:discussion}) can be
  found in} \cite{DZ91}.  To the best of our knowledge, the most
general version of \eqref{eq:99} and \eqref{eq:69} is due to
\cite{K85}, where the following result is proved 
\begin{thm}[Klaassen bounds]\label{thm:klabou}
  Let $\mu$ be some $\sigma$-finite measure. Let ${\rho}(x, y)$ be
  a measurable function such that {${\rho}(x, \cdot)$ does not
    change sign for $\mu$ almost $x\in\mathbb{R}$}. Suppose that $g$ is a
  measurable function such that
  $G(x) = \int {\rho}(x, y) g(y) \,\mu( \mathrm{d}y)+c$ is well
  defined for some $c \in \R$.  Let $X$ be a real random variable with
  density $p$ with respect to $\mu$.
  \begin{itemize}
  \item \label{item:KB1} (Klaassen upper variance bound) For all
    nonnegative measurable functions $h : \R \to \R$ such that
    $\mu \left( \left\{ x \in \R \, | \, g(x) \neq 0, \, p(x) h(x) =
        0\right\} \right) = 0$ we have
\begin{equation}
  \label{eq:klaass}
\mathrm{Var}[ G(X)] \le \mathbb{E} \left[ 
\frac{g(X)^2}{h(X)}  \left( \frac{1}{p(X)} \int{\rho(z, X)} H(z) p(z)
 \mu(\mathrm{d}z) \right)
\right] 
\end{equation}
with  $H : \R \to \R$ supposed  well-defined by
 $H(x) = \int {\rho}(x, y) h(y) \,\mu(\mathrm{d}y)$. 
\item \label{item:KB2} (Cram\'er-Rao lower variance bound) For all
  measurable functions $k : \R \to \R$ such that
  $0 < \mathbb{E}[k^2(X)] < \infty$ and $\mathbb{E}[k(X)] = 0$ we have
\begin{equation}
  \label{eq:klCR}
  \mathrm{Var}[G(X)] \ge \frac{\mathbb{E}\left[g(X) K(X)\right]^2}{\mathrm{Var}[k(X)]}
\end{equation}
where
$K(x) = \frac{1}{p(x)} \int {\rho}(z, x)k(z) p(z) \mu(\mathrm{d}z)$.
Equality in \eqref{eq:klCR} holds if and only if $G$ is linear in $k$,
$p$-almost everywhere.
  \end{itemize}
\end{thm}
Klaassen's proof of Theorem \ref{thm:klabou} relies on little more
than the Cauchy-Schwarz inequality and Fubini's theorem; it has a
slightly magical aura as little or no heuristic or context is provided
as to the best choices of test functions $h, k$ and kernel $\rho$ or
even to the nature of the weights appearing in \eqref{eq:klaass} and
\eqref{eq:klCR}.  To the best of our knowledge, all available first
order variance bounds from the literature can be obtained from either
\eqref{eq:klaass} or \eqref{eq:klCR} by choosing the appropriate test
functions $h$ or $k$ and the appropriate kernel $\rho$.  For instance,
the weights appearing in the upper bound \eqref{eq:klaass} generalize
the Stein kernel from Cacoullos' bound \eqref{eq:69} -- both in the
discrete and the continuous case. Indeed taking $H(x) = x$ when the
distribution $p$ is continuous we see that then $h(x) = 1$ and the
weight becomes $p(x)^{-1}\int \rho(z, x) z p(z) \mathrm{d} \mu(z)$
which is none other than $\tau_p(x)$. A similar argument holds as well
in the discrete case.  In the same way, taking $k(x) = x$ leads to
$K(x) = \tau_p(x)$ in \eqref{eq:klCR} and thus the lower bound in
\eqref{eq:69} follows as well.  The freedom of choice in the function
$h$ allows for much flexibility in the quality of the weights; this
fact seems somewhat under exploited in the literature.  This is
perhaps due to the rather obscure nature of Klaassen's weights, a topic
which we shall be one of the central learnings of this paper. Indeed
we shall provide a natural theoretical home for Klaassen's result, in
the framework of Stein's method.

Several variations on Klaassen's theorem have already been obtained
via techniques related to
Stein's method. We defer a proper
introduction of these techniques to Section \ref{sec:stein-diff}.  The
gist of the approach can nevertheless be understood very simply in
case the underlying distribution is standard normal.  Stein's
classical identity states that if $N \sim \mathcal{N}(0,1)$ then
\begin{equation}\label{eq:2}
  \mathbb{E}[N g(N)] = \mathbb{E}[g'(N)] \mbox{ for all } g \mbox{
    such that } \mathbb{E}[|g'(N)|]<\infty. 
\end{equation}
By the Cauchy-Schwarz inequality we immediately deduce that, for all
appropriate $g$,
\begin{align}\label{eq:22}
\mathbb{E}[g'(N)]^2 =     \mathbb{E}[N g(N)]^2 \le 
\mathbb{E}[N^2]  \mathbb{E} \left[ g(N)^2 \right] \le \mathrm{Var}[g(N)], 
\end{align}
which gives the lower bound in \eqref{eq:99}. For the upper bound, still by the Cauchy-Schwarz inequality,
\begin{align}\label{eq:27}
\mathrm{Var}[g(N)] \le    \mathbb{E}\left[\left(\int_0^Ng'(x) \mathrm{d}x
  \right)^2\right] & \le  \mathbb{E}  \left[ N \int_0^N(g'(x))^2 \mathrm{d}x
  \right] 
   =  \mathbb{E}  \left[ (g'(N))^2 \right]
\end{align}
where the last identity is a direct consequence of Stein's identity
\eqref{eq:2} applied to the function
$g(x) = \int_0^x(g'(u))^2 \mathrm{d}u$. This is the upper bound in
\eqref{eq:99}.  The idea behind this proof is due to Chen \cite{ch82}.  As
is now well known (again, we refer the reader to Section
\ref{sec:stein-diff} for references and details), Stein's identity
\eqref{eq:2} for the normal distribution can be extended to basically
any univariate (and even multivariate) distribution via a family of
objects called ``Stein operators''. This leads to a wide variety of
Stein-type integration by parts identities and it is natural to wonder
whether Chen's approach can be used to obtain generalizations of
Klaassen's theorem. First steps in this direction are detailed in
\cite{LS12a,LS16}; in particular it is seen that general lower
variance bounds are easy to obtain from generalized Stein identities
in the same way as in \eqref{eq:22}. Nevertheless, the method of proof
in \eqref{eq:27} for the upper bound cannot be generalized to
arbitrary targets and, even in cases where the method does apply, the
assumptions under which the bounds hold are quite stringent. To the
best of our knowledge, the first to obtain upper variance bounds via
properties of Stein operators is due to Saumard \cite{saumard2018weighted}, by
combining generalized Stein identities -- expressed in terms of the
Stein kernel $\tau_p(x)$ -- with Hoeffding's identity
\eqref{eq:Hoefdingcov}.  The scope of Saumard's weighted Poincar\'e
inequalities is, nevertheless, limited and a general result such as
Klaassen's is, to this date, not available in the literature.

The main contributions of this paper can be categorized in  two types:
\begin{itemize}
\item \emph{Covariance identities and inequalities.}  The first main
  contribution of this paper is a generalization of Klaassen's
  variance bounds from Theorem \ref{thm:klabou} to covariance
  inequalities of arbitrary functionals of arbitrary {univariate}
  targets under minimal assumptions (see Theorems \ref{prop:cramerao}
  and \ref{thm:klaassenbounds}). Our results hereby therefore also
  contains basically the entire literature on the topic. Moreover, the
  weights that appear in our bounds bear a clear and natural
  interpretation in terms of Stein operators which allow for easy
  computation for a wide variety of targets, as illustrated in the
  different examples we tackle as well as in Tables \ref{tab:Ord1},
  \ref{tab:Ord2} and \ref{tab:Pearson} in which we provide explicit
  variance bounds for univariate target distributions belonging to the
  classical integrated Pearson and Ord families (see Example
  \ref{def:pearsord} for a definition). In particular, Klaassen's
  bounds now arise naturally in this setting.
\item \emph{Stein operators and their properties.}  The second main
  contribution of the paper lies in our method of proof, which
  contributes to the theory of Stein operators
  themselves. Specifically, we obtain several new probabilistic
  representations of inverse Stein operators (a.k.a.\ solutions to
  Stein equations) which open the way to a wealth of new manipulations
  which where hitherto unavailable. These representations also lead to
  new interpretations and ultimately new handles on several quantities
  which are crucial to the theory surrounding Stein's method
({such as} Stein kernels, Stein equations, Stein factors, {and} Stein
bounds). 
Finally the various objects we identify provide natural
connections with other topics of interest, including the well-known
connection with Sturm-Liouville theory already identified in
\cite{DZ91}. 
\end{itemize}

The paper is organised as follows.  Section \ref{sec:stein-diff}
contains the theoretical foundations of the paper. In Section
\ref{subsec:def} we recall the theory of canonical and standardized
Stein operators introduced in \cite{LRS16} and introduce a (new)
notion of inverse Stein operator (Definition
\ref{def:can_inverse}). We also identify minimal conditions under
which Stein-type probabilistic integration by parts formulas hold (see
Lemmas \ref{lma:skadj} and \ref{lem:stIBP1}). In Section
\ref{sec:three-repr-inverse}
we provide the representation formulas for the inverse Stein operator
(Lemmas \ref{lem:represnt} and \ref{lem:can_inverserep2}). In Section
\ref{sec:suff-cond} we clarify the conditions on the test functions
under which the different identities hold, and provide bridges with
the classical assumptions in the literature. Section
\ref{subsec:inverseOp} contains bounds on the solutions to the Stein
equations. Section \ref{sec:variance-expansion} contains the
covariance identities and inequalities. After re-interpreting
Hoeffding's identity \eqref{eq:Hoefdingcov} we obtain general and
flexible lower and upper \emph{covariance} bounds (Proposition
\ref{prop:cramerao} and Theorem \ref{thm:klaassenbounds}). We then
deduce Klaassen's bounds (Corollary \ref{cor:klaaaaaassen}) and
provide examples for several concrete distributions, with more
examples deferred to
the three tables mentioned above. Finally a discussion is provided in
Section \ref{sec:discussion}, wherein several examples are treated and
connections with other theories are established, for instance the
Brascamp-Lieb inequality (Corollary \ref{cor:brascamplieb}) and Menz
and Otto's asymmetric Brascamp-Lieb inequality (Corollary
\ref{cor:asbraslieb}), as well as the link with an eigenfunction
problem which can be seen as an extended Sturm-Liouville problem.
{The proofs from Section \ref{sec:suff-cond} are technical and
  postponed to the appendix \ref{sec:some-non-essential}.}
  
\section{Stein differentiation}
\label{sec:stein-diff}

Stein's method consists in a collection of techniques for
distributional approximation that was originally developed for normal
approximation in \cite{S72} and for Poisson approximation in
\cite{C75}; {for expositions see the books} 
\cite{Stein1986,BC05,BC05b,ChGoSh11,NP11} 
  {and}  the review papers
\cite{Re04,Ro11,chatterjee2014short}. Outside the Gaussian and Poisson
frameworks, there exist several non-equivalent general theories
allowing to setup Stein's method for large swaths of probability
distributions, of which we single out the papers
\cite{CS11,Do14,upadhye2014stein} for univariate distributions under
analytical assumptions, \cite{arras2017stein, arras2018stein} for
infinitely divisible distributions, \cite{barbour2018multivariate} for
discrete multivariate distributions, and
\cite{mackey2016multivariate,gorham2016measuring,gorham2017measuring}
as well as \cite{fang2018multivariate} for multivariate densities
under diffusive assumptions.

The backbone of the present paper consists in the approach from
\cite{ley2015distances,LRS16,MRS18}. Before introducing these results,
we fix the notations.  Let $\mathcal{X}{\in \mathcal{B}(\R)}$ and equip it with
some $\sigma$-algebra $\mathcal{A}$ and $\sigma$-finite measure
$\mu$. Let $X$ be a random variable on $\mathcal{X}$, with {induced probability
measure $\mathbb{P}^X$} which is absolutely continuous with respect to $\mu$; we
denote {by} $p$ the corresponding probability density, and its support
by $\mathcal{S}(p) = \left\{ x \in \mathcal{X} : p(x)>0\right\}$.  As
usual, $L^1(p)$ is the collection of all real valued functions $f$
such that $\mathbb{E}|f(X)| < \infty$. We sometimes call the
expectation under $p$ the $p$-mean. Although we could in principle
keep the discussion to come very general, in order to make the paper
more concrete and readable we shall restrict our attention to
distributions satisfying the following Assumption.

\

\noindent \textbf{Assumption A.} The measure $\mu$ is either the
counting measure on $\mathcal{X} = \Z$ or the Lebesgue measure on
$\mathcal{X} = \R$. If $\mu$ is the counting measure then there exist
$a {<} b \in \Z \cup \left\{-\infty, \infty \right\}$ such that
$\mathcal{S}(p) =  [a, b]\cap \Z$.  If $\mu$ is the Lebesgue measure
then there exist
$a, b \in \R\left\{-\infty, \infty \right\}$ such  that 
${\mathcal{S}(p)}^{\mathrm{o}} = ]a, b[$ and 
$\overline{\mathcal{S}(p)} = [a, b]$.
{Moreover, the measure $\mu$ is not point mass.}

\

{Here not allowing point mass much simplifies the presentation. Stein's method for point mass is available in \cite{reinert1995weak}.} 

    


\medskip Let $\ell \in \left\{ -1, 0, 1 \right\}$. In the sequel we
shall restrict our attention to the following three derivative-type
operators: 
\begin{eqnarray*}
  \Delta^{\ell}f(x) &= \left\{ 
\begin{array}{l l } 
  f'(x), &  \mbox{ if }\ell = 0;\\
  \frac{1}{\ell} (f(x+\ell) - f(x)) & \mbox{ if } \ell \in \{-1, +1\}, \\
 \end{array} \right. 
\end{eqnarray*} 
with $f'(x)$ the weak derivative defined Lebesgue almost everywhere,
$\Delta^{+1} ({\equiv} \Delta^+)$ the classical forward difference and
$\Delta^{-1} ({\equiv} \Delta^-)$ the classical backward
difference. Whenever $\ell=0$ we take $\mu$ as the Lebesgue measure
and speak of the \emph{continuous case}; whenever
$\ell \in \left\{-1, 1 \right\}$ we take $\mu$ as the counting measure
and speak of the \emph{discrete case}. There are two choices of
derivatives in the discrete case, only one in the continuous
case. {We let $\mathrm{dom}(\Delta^{\ell})$ denote the collection
  of functions $f : \R \to \R$ such that $\Delta^{\ell}f(x)$ exists
  and is finite $\mu$-almost surely. In the case $\ell=0$, this
  corresponds to all absolutely continuous functions; in the case
  $\ell = \pm1$ the domain is the collection of all functions on
  $\Z$.}  For ease of reference we note that, if
{$f \in \mathrm{dom}(\Delta^{\ell})$ is such that
  $\Delta^{\ell} f\, \mathbb{I}[a, b] \in L^1(\mu)$} then, for all
{$c, d$} such that $a \le c \le d\le b$ we have
\begin{align*}
  \int_c^d \Delta^\ell f(x)\, \mu(\mathrm{d}x) 
    = 
                                                \begin{cases}
                                                                                                    \int_{c}^d
                                                  f'(x) \mathrm{d}x  =
                                                  f(d) - f(c)  &
                                                  \mbox{ if } \ell =
                                                  0 \\
                                                  \sum_{j=c}^d
                                                  \Delta^-f(x)  =
                                                  f(d) - f(c-1)  &
                                                  \mbox{ if } \ell =
                                                  -1 \\
                                                  \sum_{j=c}^d
                                                  \Delta^+f(x)  =
                                                  f(d+1) - f(c)  &
                                                  \mbox{ if } \ell =
                                                  +1 
                                                \end{cases}                                 
\end{align*}
which we summarize as
\begin{equation}
  \label{eq:30}
  \int_c^d \Delta^\ell f(x) \mu(\mathrm{d}x) = 
   f(d+a_{\ell}) -  f(c-b_{\ell})
\end{equation}
where
\begin{equation}
  \label{eq:31}
 a_{\ell} 
  = \mathbb{I}[\ell=1] \mbox{ and } b_{\ell} =
\mathbb{I}[\ell = -1].
\end{equation}
{We stress the fact that the values at $c, d$ are understood as limits
if either is infinite.}

      \subsection{Stein operators and Stein equations} \label{subsec:def}
      
      Our first definitions come from \cite{LRS16}. {We first define
      $\mathrm{dom}(p, \Delta^{\ell})$ as the collection of
      $f : \R \to \R$ such that
      $f\,p \in \mathrm{dom}(\Delta^{\ell})$.}

\begin{defn}[Canonical Stein operators] \label{def:stei_op} Let
  $f \in \mathrm{dom}(p, \Delta^{\ell})$ and consider the linear
  operator $f \mapsto \mathcal{T}_p^{\ell}f$ defined as
  $$ \mathcal{T}_p^{\ell} f(x) = \frac{\Delta^\ell (f(x) p(x))}{p(x)}$$
  for all $x \in \mathcal{S}(p)$ and $ \mathcal{T}_p^{\ell} f(x) = 0$
  for $x \notin \mathcal{S}(p)$. {The} operator $\mathcal{T}_p^{\ell}$ is {called} 
  the canonical ($\ell$-)Stein operator of $p$.  The cases $\ell=1$
  and $\ell = -1$ provide the {forward} and {backward} Stein
  operators, denoted {by} $\mathcal{T}_p^+$ and $\mathcal{T}_p^-$,
  respectively; the case $\ell=0$ provides the differential Stein
  operator denoted by $\mathcal{T}_p$. 
\end{defn}

\noindent {To describe the domain and the range of $\mathcal{T}_p^{\ell}$ we
  introduce the following sets of functions:}
\begin{eqnarray*}
  \mathcal{F}^{(0)}(p) &=& \bigg\{ f \in L^1 (p) : \mathbb{E}[f(X)]= 0 \bigg\}; \\
  \mathcal{F}^{(1)}_\ell (p) &=& \left\{ f \in \mathrm{dom}(p, \Delta^\ell) : 
  \Delta^\ell(f p ) {\mathbb{I}[\mathcal{S}(p)]} \in L^1(\mu) \mbox{ and }
  \int_{\mathcal{S}(p)} \Delta^\ell (fp)(x)  \, \mu(\mathrm{d}x) 
  = \mathbb{E}\left[  \mathcal{T}_p^{\ell}f(X) \right]= 0 \right\}.
\end{eqnarray*}
{We draw the reader's attention to the fact that the second condition
in the definition of $\mathcal{F}^{(1)}_\ell (p)$ can be rewritten
\begin{equation*}
  f(b+a_{\ell})  p(b+a_{\ell}) =   f(a-b_{\ell})  p(a-b_{\ell}).
\end{equation*}

}
{The next lemma, which  follows immediately from the
  definition of $\mathcal{T}_p^{\ell} f$ and of the different sets of
  functions,  shows why $\mathcal{F}^{(1)}_\ell (p) $ is called 
  the {\it{canonical Stein class}}.}

\begin{lma}[Canonical Stein class]\label{def:can_class} 
For $f \in\mathcal{F}^{(1)}_\ell (p) $, $
\mathcal{T}_p^{\ell} f \in \mathcal{F}^{(0)}(p)$ . 
\end{lma}  
Crucially for the results in this paper, for all
{$f\in \mathrm{dom}(\Delta^\ell)$,
  {$g\in \mathrm{dom}(\Delta^{-\ell})$} such that
  $f(\cdot)g(\cdot - \ell) \in \mathrm{dom}(\Delta^\ell)$} the
operators $\Delta^{\ell}$ satisfy the product rule
\begin{equation}
  \label{eq:productrule} \Delta^{\ell}\big(f(x)g(x-\ell)\big) =( \Delta^{\ell} f(x) ) g
  (x) + f(x) \Delta^{-\ell}g(x)\end{equation}
for all
$\ell \in \left\{ -1, 0, 1 \right\}$. {This product rule leads to an
  integration by parts  (IBP) formula {(a.k.a.\ Abel-type  summation
    formula)} as follows.}

\begin{lma}[Stein IBP formula - version 1] \label{lma:skadj} For all
  {$f \in \mathrm{dom}(p, \Delta^\ell)$,
    $g \in \mathrm{dom}(\Delta^{-\ell})$} such that (i)
  $f(\cdot)g(\cdot - \ell) \in \mathcal{F}_{\ell}^{(1)}(p)$ and (ii)
  $f(\cdot) \Delta^{-\ell}g(\cdot) \in L^1(p)$ we have
\begin{equation}
  \label{eq:3}
\mathbb{E}\left[ (\mathcal{T}_p^{\ell}f(X) ) g(X) \right] = -
\mathbb{E} \left[ f(X) \Delta^{-\ell} g(X) \right]. 
\end{equation}
\end{lma}
\begin{proof}
 { Under the stated assumptions, we can apply  \eqref{eq:productrule}
  to get 
\begin{equation}
  \label{eq:90}
  \mathcal{T}_p^{\ell} \big(f(x) g(x-\ell)\big) = ( \mathcal{T}_p^{\ell}f(x) ) 
g(x) + f(x) (\Delta^{-\ell}g(x))
\end{equation}
for all $x \in \mathcal{S}(p)$.  Condition~(i) in the statement
guarantees that the left hand side (l.h.s.) of \eqref{eq:90} has mean
0, while condition~(ii) guarantees that we can separate the
expectation of the sum on the right hand side (r.h.s.) into the sum of
the individual expectations.}
\end{proof}

A natural interpretation of \eqref{eq:3} is that operator
$\mathcal{T}_p^{\ell}$ is, in some sense to be made precise, the
\emph{skew-adjoint} operator to $\Delta^{-\ell}$ with respect to the
scalar product
$\left\langle f, g \right\rangle = \mathbb{E} \left[ f(X) g(X)
\right]$; this {provides a supplementary justification} to the use
of the terminology ``canonical'' for operator
$\mathcal{T}_p^{\ell}$. {We discuss a consequence of this
  interpretation in Section~\ref{sec:discussion}. The conditions under
  which Lemma \ref{lma:skadj} holds are all but transparent. We
  clarify these assumptions in Section \ref{sec:suff-cond}.}
For more details on Stein class and operators, we refer to
\cite{LRS16} for the construction in an abstract setting,
\cite{ley2015distances} for the construction in the continuous setting
(i.e.\ $\ell = 0$) and \cite{ERSvb2} for the construction in the
discrete setting (i.e.\ $\ell \in \left\{ -1, 1
\right\}$). Multivariate extensions are developed  in \cite{MRS18}.


\medskip
The fundamental stepping stone for our theory is an
inverse of the canonical operator $\mathcal{T}_p^{\ell}$ provided in
the next definition.

\begin{defn}[Canonical pseudo inverse Stein
  operator]\label{def:can_inverse}
  Let $\ell \in \left\{ -1, 0, 1 \right\}$ and recall the notations
  $a_{\ell}, b_{\ell}$ from \eqref{eq:31}. The {canonical
    pseudo-inverse Stein operator} $ \mathcal{L}_p^{\ell}$ for the
  operator $\mathcal{T}_p^{\ell}$ is defined, for $h \in L^1(p)$, as
 \begin{equation}
    \label{eq:53}
    \mathcal{L}_p^{\ell}h(x) = \frac{1}{p(x)} \int_a^{x-a_{\ell}} (h(u)  -
    \mathbb{E}[h(X)])p(u) \mu(\mathrm{d}u) = \frac{1}{p(x)}
    \int_{x+b_{\ell}}^{b} (
    \mathbb{E}[h(X)]-h(u))p(u) \mu(\mathrm{d}u)  
\end{equation}
for all $x \in \mathcal{S}(p)$ and 
    $ \mathcal{L}_p^{\ell}h(x) = 0$ for all
    $x \notin \mathcal{S}(p)$.
  \end{defn}

  Equality between the second and third expressions in \eqref{eq:53}
  is justified because $h \in L^1(p)$ so that the integral of
  $h(\cdot) - \mathbb{E}[h(X)]$ over the whole support cancels out.
  For ease of reference we detail $\mathcal{L}_p^{\ell}$  in the
  three cases that interest us:
\begin{align*}
& \mathcal{L}_p^0h(x) = \frac{1}{p(x)} \int_a^x (h(u)  -
  \mathbb{E}[h(X)])p(u) \mathrm{d}u   = \frac{1}{p(x)} \int_x^b (
  \mathbb{E}[h(X)]-h(u))p(u) \mathrm{d}u   &(\ell=0) \nonumber\\
& \mathcal{L}_p^-h(x) =
\frac{1}{p(x)} \sum_{j=a}^{x} (h(j) - \mathbb{E}[h(X)]) p(j)  =
                    \frac{1}{p(x)} \sum_{j=x+1}^{b}
                    (\mathbb{E}[h(X)]-h(j)) p(j) &(\ell=-1)\nonumber \\
 & \mathcal{L}_p^+h(x) = \frac{1}{p(x)}
\sum_{j=a}^{x-1} (h(j) - \mathbb{E}[h(X)])p(j)   = \frac{1}{p(x)}
\sum_{j=x}^{b} (\mathbb{E}[h(X)]-h(j))p(j)  &(\ell=1). \nonumber
\end{align*}
Note that $\mathcal{L}_p^-h(b) = 0$ but
$\mathcal{L}_p^-h(a) = h(a) - \mathbb{E}[h(X)]$ and, conversely,
$\mathcal{L}_p^+h(a) = 0$ but
$\mathcal{L}_p^+h(b) = \mathbb{E}[h(X)] - h(b)$.  The denomination \emph{pseudo-inverse-Stein operator} for
$\mathcal{L}_p^{\ell}$ is justified by the following lemma whose proof
is immediate.  {\begin{lma}\label{lem:inverse} For any
    $h \in L^1(p)$,
    $\mathcal{L}^{\ell}_p h \in
    \mathcal{F}_{\ell}^{(1)}(p)$. Moreover, (i) for all $h \in L^1(p)$
    we have
    $ \mathcal{T}_{p}^{\ell}\mathcal{L}^{\ell}_ph(x) = h(x) -
    \mathbb{E}[h(X)]$ at all $x \in \mathcal{S}(p)$ and (ii) for all
    $f \in \mathcal{F}^{(1)}_{\ell}(p)$ we have
    $ \mathcal{L}^{\ell}_p\mathcal{T}_{p}^{\ell}f(x) = f(x) -
    \frac{f(a^+-b_{\ell})p(a^+-b_{\ell})}{p(x)} = f(x) -
    \frac{f(b^-+a_{\ell})p(b^-+a_{\ell})}{p(x)}$ at all
    $x \in \mathcal{S}(p)$. {Operator $\mathcal{L}_p^{\ell}$ is
      invertible (with inverse $\mathcal{T}_p^{\ell}$) on the subclass
      of functions in $\mathcal{F}^{(0)}(p) \cap \mathcal{F}^{(1)}(p)$
      which, moreover, satisfy $f(b^{-}+a_{\ell})p(b^{-}+a_{\ell}) =f(a^{+}-b_{\ell})p(a^+-b_{\ell})= 0$.}
\end{lma}
}

 Starting from
  \eqref{eq:90} we postulate the next definition.
\begin{defn}[Standardizations of the canonical
  operator]\label{def:stand-canon-oper}
  Fix $\ell \in \left\{ -1, 0, 1 \right\}$ and $\eta \in L^1(p)$. The
  {$\eta$-standardized Stein operator} is
\begin{align}
  \label{eq:33}
  \mathcal{A}_p^{\ell, \eta}g(x) & = \mathcal{T}^{\ell}_p\big(\mathcal{L}^{\ell}_p\eta(\cdot)
                                   g(\cdot-\ell)\big)(x) =  \big(\eta(x) - \mathbb{E}[\eta(X)]\big)  g(x) +
                                   \mathcal{L}^{\ell}_p\eta(x) 
                                   \big(\Delta^{-\ell}g(x)\big) 
\end{align}
acting on the collection $\mathcal{F}(\mathcal{A}_p^{\ell, \eta})$ of
test functions $g$ such that
{$ \mathcal{L}^{\ell}_p\eta(\cdot)g(\cdot - \ell) \in
  \mathcal{F}_{\ell}^{(1)}(p)$} and
$(\mathcal{L}^{\ell}_p\eta) \Delta^{-\ell}g \in L^1(p)$.
\end{defn}
\begin{rmk}
  The conditions appearing in the definition of
  $\mathcal{F}(\mathcal{A}_p^{\ell, \eta})$ are tailored to ensure that all
  identities and manipulations follow immediately. For instance, the
  requirement that
  {$\mathcal{L}^{\ell}_p\eta(\cdot)g(\cdot - \ell) \in
    \mathcal{F}_{\ell}^{(1)}(p)$} in the definition of
  $\mathcal{F}(\mathcal{A}_p^{\ell, \eta})$ guarantees that the resulting
  functions {$ \mathcal{A}_p^{\ell, \eta}g(x) $} have $p$-mean 0 and the
  condition $(\mathcal{L}^{\ell}_p\eta) \Delta^{-\ell}g \in L^1(p)$
  guarantees that the expectations of the individual summands on the
  r.h.s.\ of \eqref{eq:33} exist.
  Again, our assumptions are not transparent; we discuss them in
  detail in Section~\ref{sec:suff-cond}.
\end{rmk}

The final ingredient {for Stein differentiation} is the \emph{Stein
  equation}: {\begin{defn}[Stein equation] {Fix
      $\ell \in \left\{ -1, 0, 1 \right\}$ and $\eta \in L^1(p)$.} For
    ${h} \in L^1(p)$,
    the $\mathcal{A}_p^{\ell, \eta}$-Stein equation {for $h$} is the
    functional equation
    $\mathcal{A}_p^{\ell, \eta}g(x) = h(x) - \mathbb{E}[h(X)], x \in
    \mathcal{S}(p)$, i.e.\
\begin{equation}
  \label{eq:43}
  \big(\eta(x) - \mathbb{E}[\eta(X)]\big)  g(x) +
  \mathcal{L}^{\ell}_p\eta(x) 
  \big(\Delta^{-\ell}g(x)\big)  = h(x) - \mathbb{E}[h(X)], \,  x \in
  \mathcal{S}(p). 
\end{equation}
A solution to the Stein equation  is any function
$g \in \mathcal{F}(\mathcal{A}_p^{\ell, \eta})$ which satisfies
\eqref{eq:43} for all $x \in \mathcal{S}(p)$.
\end{defn}
}
Our notations lead immediately to the next result.
\begin{lma}[Solution to the Stein equation]\label{lem:steinsol} 
{Fix $\eta \in L^1(p)$.} The {Stein equation} \eqref{eq:43} {for $h \in L^1(p)$ is solved by}
\begin{equation}
  \label{eq:8}
 g_h^{p, \ell, \eta} (x) 
 = \frac{\mathcal{L}_p^{\ell} h(x+ \ell)}{\mathcal{L}_p^{\ell} 
 \eta(x+\ell)}
\end{equation}
with the convention that $g_h^{p, \ell, \eta}(x) = 0$ for all $x+\ell$ outside of
$\mathcal{S}(p)$.
\end{lma}
\begin{proof}
With $g = g_h^{p, \ell, \eta}$, 
\begin{align*}
  \mathcal{A}_p^{\ell, \eta}g (x) 
 &= 
  \mathcal{T}^{\ell}_p\big(\mathcal{L}^{\ell}_p\eta(\cdot)
                      g(\cdot-\ell)\big)(x) 
                      =
                     \mathcal{T}^{\ell}_p\left(\mathcal{L}^{\ell}_p\eta(\cdot)
                      \frac{\mathcal{L}_p^{\ell} h(\cdot)}{\mathcal{L}_p^{\ell} 
 \eta(\cdot)} \right)(x) 
   =
                     \mathcal{T}^{\ell}_p\left(
                     \mathcal{L}_p^{\ell} h(\cdot) \right)(x) \\ 
 &= h(x) - \mathbb{E}[h(X)]
 \end{align*} 
 using Lemma \ref{lem:inverse} for the last step. Hence \eqref{eq:43}
 is satisfied for all $x \in \mathcal{S}(p)$. Since, by construction,
 $g \in \mathcal{F}(\mathcal{A}_{p }^{\ell, \eta})$, the claim follows.
\end{proof}

When the context is clear then we drop the superscripts and the
subscript in $g$ of \eqref{eq:8}.  Before proceeding we provide {two}
examples. {The notation $\mathrm{Id}$ refers to the identity function
  $x \rightarrow \mathrm{Id}(x) = x$.}

\begin{exm}[Binomial distribution]\label{ex:binom1}
  Let $p(x) = \binom{n}{x}\theta^x(1-\theta)^{n-x}$ be the binomial
  density with parameters $(n, \theta)$ and
  $\mathcal{S}(p) = [0,n]\cap \N$; assume that $0 < \theta < 1$.
  Stein's method for the binomial distribution was first developed in
  \cite{ehm1991binomial} {using $\Delta^{-}$}; see also
  \cite{soon1996binomial, Ho04}.

  \medskip Picking $\ell = 1$, the class $\mathcal{F}_+^{(1)}(p)$
  consists of functions $f : \Z \to \R$ which are bounded on
  $\mathcal{S}(p)$ and $f(0) = 0$. Fixing $\eta(x) = x-n\theta $ gives
  $ \mathcal{L}_{\mathrm{bin}(n, \theta)}^+ \eta(x) = -(1-\theta)x$
  leading to
    \begin{align}
      \label{eq:34}
      \mathcal{A}^{+, \mathrm{Id}}_{\mathrm{bin}(n, \theta)}g(x) =  (x-n\theta)g(x) - (1-\theta)x\Delta^-g(x)
    \end{align}
    with corresponding class
    $ \mathcal{F}\left(\mathcal{A}^{+, \mathrm{Id}}_{\mathrm{bin}(n, \theta)}\right)$ which contains
    all functions $g : \Z \to \R$. 
    The solution to the
    $\mathcal{A}^{+, \mathrm{Id}}_{\mathrm{bin}(n, \theta)}$-Stein
    equation (see \eqref{eq:43}) is
     \begin{equation*}
       g^+(x) = \frac{-1}{(1-\theta)(x+1)p(x+1)} \sum_{j=0}^{x} (h(j) -
       \mathbb{E}[h(X)]) p(j) \mbox{ for all }  0 \le x \le n-1 
     \end{equation*}
     and $g^+(n) = 0$.

     \medskip Picking $\ell = -1$, the class $\mathcal{F}_-^{(1)}(p)$
     consists of functions $f : \Z \to \R$ which are bounded on
     $\mathcal{S}(p)$ and such that $f(n) = 0$. Again fixing
     $\eta(x) = x-n\theta$ gives
     $\mathcal{L}_{\mathrm{bin}(n, \theta)}^- \eta(x) = -\theta(n-x)$
     leading to
    \begin{align}
      \mathcal{A}^{-, \mathrm{Id}}_{\mathrm{bin}(n, \theta)}g(x) = 
      (x-n\theta)g(x) -
       \theta(n-x)\Delta^+g(x) 
       \nonumber
    \end{align}
    acting on the same class as \eqref{eq:34}. The solution to the
    $\mathcal{A}^{-, \mathrm{Id}}_{\mathrm{bin}(n, \theta)}$-Stein equation is
     \begin{equation*}
 g^-(x) = \frac{-1}{\theta(n-(x-1))p(x-1)} \sum_{j=0}^{x-1} (h(j) -
       \mathbb{E}[h(X)]) p(j) \mbox{ for all }  1 \le x \le n
     \end{equation*}
     and $g^-(0) = 0$. The function $-g^-$ is studied in
     \cite{ehm1991binomial} where bounds on $\| \Delta^-g^-\|$ are
     provided (see equation (10) in that paper); see also Section
     \ref{subsec:inverseOp} where bounds on $\|g^-\|$ are provided.
 
\end{exm}

\begin{exm}[Beta distribution]\label{ex:beta1}
  Let $p(x) = x^{\alpha-1}(1-x)^{\beta-1}/B(\alpha, \beta)$ be the
  beta density with parameters $(\alpha, \beta)$ and
  $\mathcal{S}(p) = (0,1)$. Stein's method for the beta distribution
  was developed in \cite{goldstein2013stein,Do14} {using the Stein operator 
  $\mathcal{A}f(x) = x(1-x) f'(x) + (\alpha(1-x) - \beta x) f(x)$}.  In our notations,
  we have $\ell = 0$ and $\mathcal{F}^{{(1)}}_0(p)$ consists of functions
  $f: \R \to \R$ such that $f(0^+)p(0^+) = f(1^-)p(1^-)$ and
  $|(f p)'|$ is Lebesgue integrable on $[0,1]$. Fixing
  $\eta(x) = x - \frac{\alpha}{\alpha+\beta}$ gives
  $ \mathcal{L}_{\mathrm{beta}(\alpha, \beta)}\eta(x) =
  -\frac{x(1-x)}{\alpha + \beta}$ leading to the operator
    \begin{align*}
      \mathcal{A}^{\mathrm{Id}}_{\mathrm{Beta}(\alpha, \beta)}g(x) =
      \Big(x-\frac{\alpha}{\alpha+\beta}\Big) g(x) -\frac{
      x(1-x)}{\alpha+\beta}  g'(x)
    \end{align*}
    with {domain}
    $
    \mathcal{F}\left(\mathcal{A}^{\mathrm{Id}}_{\mathrm{Beta}(\alpha,
        \beta)}\right)$ {the set of} differentiable functions
    $g : \R \to \R$ such that $x(1-x) g(x) \in \mathcal{F}^{{(1)}}_0(p)$
    and $x(1-x) g'(x) \in L^1(p)$. The solution to the
    $ \mathcal{A}^{\mathrm{Id}}_{\mathrm{Beta}(\alpha, \beta)}$-Stein
    equation is
     \begin{equation*}
       g(x) = \frac{-(\alpha+\beta)}{x(1-x)p(x)} \int_0^x (h(u) -
       \mathbb{E}[h(X)]) p(u)du, \quad x \in (0,1). 
     \end{equation*}
     {The operator
       $ \mathcal{A}^{\mathrm{Id}}_{\mathrm{Beta}(\alpha, \beta)}f$}
     is, up to multiplication by $\alpha+\beta$, the classical Stein
     operator $\mathcal{A}f$ for the beta density, see
     \cite{goldstein2013stein, Do14} for details and bounds on
     solutions and their derivatives. See also Section
     \ref{subsec:inverseOp} where bounds on $\|g\|$ are provided.
\end{exm}

In order to propose a more general example, we recall the concept of a
Stein kernel, here extended to continuous and discrete distributions
alike.
  \begin{defn}[The Stein kernel]\label{def:kernel}
    Let $X\sim p$ have finite mean. The ($\ell$-)Stein kernel of $X$ (or
    of $p$) is the function
    \begin{equation*}
      \tau_p^{\ell}(x) = -\mathcal{L}_p^{\ell}(\mathrm{Id})(x). 
    \end{equation*}
Metonymously, we refer to the random variable $\tau_p^{\ell}(X)$ as
the ($\ell$-)Stein kernel of $X. $
\end{defn}

\begin{rmk}
  The function $\tau_p^{\ell}(\cdot)$ is studied in detail for
  $\ell = 0$ in \cite[Lecture VI]{Stein1986}. {This function is}
  particularly useful for Pearson (and discrete Pearson a.k.a.\ Ord)
  distributions which are characterized by the fact that their Stein
  kernel $\tau_p^{\ell}$ is a second degree polynomial, see Example
  \ref{def:pearsord}. For more on this topic, {we also refer to
    forthcoming \cite{ERSvb2} as well as
    \cite{courtade2017existence,fathi2018stein, fathi2018higher}
    wherein important
       contributions to the theory of Stein
       kernels are provided in a multivariate setting.  }
\end{rmk}

{The next example {gives some ($\ell$-)Stein kernels, exploiting}
  the fact that if the mean of $X$ is $\nu$, then
  $ \mathcal{L}_p^{\ell}(\mathrm{Id})(x) =
  \mathcal{L}_p^{\ell}(\mathrm{Id} - \nu)(x).$ }

\begin{exm}\label{ex:binom2} 
If $X \sim \mathrm{Bin}(n, \theta)$ then {using $\eta(x) = x - n \theta$, Example \ref{ex:binom1}} gives 
$\tau_{\mathrm{bin}(n, \theta)}^{+}(x) = (1-\theta)x$ and
$\tau_{\mathrm{bin}(n, \theta)}^{-}(x) = \theta(n-x)$. If $X \sim
\mathrm{Beta}(\alpha, \beta)$ then {Example \ref{ex:beta1} with $\eta(x) = x - \frac{\alpha}{\alpha+\beta} $ gives} $\tau_{\mathrm{Beta}(\alpha,
  \beta)}^0(x) = x(1-x)/(\alpha+\beta)$.   
\end{exm}

\begin{exm}[A general example]\label{ex:genera}
  Let $p$ satisfy Assumption A and suppose that it has finite mean
  $\nu$.  Fixing $\eta(x) = x - \nu$, operator \eqref{eq:33} becomes
    \begin{align*}
      \mathcal{A}_p^{\tau_p^{\ell}}g(x) =
      \big(x-\nu\big) g(x) - \tau_p^{\ell}(x) \Delta^{-\ell}g(x)
    \end{align*}
    with corresponding class
    $ \mathcal{F}\bigg(\mathcal{A}_p^{\tau_p^{\ell}}\bigg)$ which
    contains all functions $g : \R \to \R$ such that
    $\tau_p^{\ell}({\cdot}) g ({\cdot -\ell})\in
    \mathcal{F}^{(1)}_{\ell}(p)$ and
    $\tau_p^{\ell}\Delta^{-\ell}g \in L^1(p)$.  Again, we stress that
    such conditions are clarified in Section
    \ref{sec:suff-cond}. {Using Lemma \ref{lem:steinsol}, the}
    solution to the $ \mathcal{A}_p^{\tau_p}$ Stein equation is
     \begin{equation*}
       g_{\mathrm{Id}}^{p, \ell,h}(x) = \frac{-\mathcal{L}_p^{\ell}h(x
         {+ \ell} )}{\tau_p^{\ell}(x{+ \ell})}. %
\end{equation*}
Bounds on $\|g\|$ are provided in Section \ref{sec:suff-cond}.
Stein's method based on $\mathcal{A}_p^{\tau_p^{\ell}}$ is already
available in several important subcases, e.g.\ in
\cite{S01,KuTu11, Do14} for continuous distributions.
\end{exm}

The construction is tailored to ensure that all operators have mean 0
over the entire classes of functions on which they are defined. We
immediately deduce the following family of Stein integration by parts
formulas:
\begin{lma}[Stein IBP formula - version 2]
\label{lem:stIBP1}
Let $X \sim p$. Then
 \begin{align}
   \label{eq:21}
  & \mathbb{E} \left[ - {{ \{}}  \mathcal{L}_p^{\ell}f(X) {{ \}}}\Delta^{-\ell}g(X)
    \right] = \mathbb{E} 
  \left[( f(X) { - \mathbb{E}[f(X)]}) g(X) \right]
 \end{align}
 for all $f \in L^1(p)$ and all $g \in \mathrm{dom}(\Delta^{-\ell})$
 such that $\mathcal{L}_p^{\ell}f(\cdot)g(\cdot - \ell) \in \mathcal{F}_{\ell}^{(1)}(p)$
 and $ {\mathcal{L}_p^{\ell} f(\cdot) \Delta^{-\ell} g(\cdot)} \in
 L^1(p)$. 
\end{lma}

\begin{proof}
  Identity \eqref{eq:21} follows directly from the Stein product rule
  in \cite[Theorem 3.24]{LRS16} or by using the fact that expectations
  of the operators {in} \eqref{eq:33} are equal to 0.  
\end{proof}

We stress the fact that in the formulation of Lemma~\ref{lem:stIBP1}
the test functions $f$ and $g$ do not play a symmetric role.  If
$g\in L^1(p)$  then the right hand side of
\eqref{eq:21} is the covariance $\mathrm{Cov}(f(X), g(X))$. We shall
use this heavily in our future developments. Similarly as for Lemma
\ref{lma:skadj}, the conditions under which Lemma~\ref{lem:stIBP1}
applies are not transparent in their present form.
{In} Section \ref{sec:suff-cond} various explicit sets of
conditions {are provided} under which the IBP \eqref{eq:21} is applicable.

\subsection{Representations of the inverse Stein operator}
\label{sec:three-repr-inverse}

This
section contains the first main results of the paper, namely
probabilistic representations for this operator. Such representations
are extremely useful for manipulations of the operators.
We start with a simple rewriting {of $\mathcal{L}_p^{\ell}$}. Given
$\ell \in \left\{ -1, 0, 1 \right\}$, recall the notation
$a_{\ell} = \mathbb{I}[\ell = 1]$ and define
 \begin{equation}
   \label{eq:kerneldeff}
   \chi^{\ell}(x, y) = 
   \mathbb{I}{[x \le y - a_\ell]}. 
 \end{equation}
 Such generalized indicator functions particularize, in the three
 cases that interest us, to $\chi^0(x, y) = \mathbb{I}[x \le y]$
 $\big(\ell = 0\big)$, $\chi^-(x, y) = \mathbb{I}[{x \le y}]$
 $\big(\ell = -1\big)$ and $ \chi^+(x, y) = \mathbb{I}[{x < y}] $
 $\big(\ell = 1\big)$.  Their properties lead to some form of
 ``calculus'' which shall be useful in the sequel. 
 \begin{lma}[Chi calculation rules] \label{lem:chirules} {The function
     $\chi^{\ell} (x,y)$ is non-increasing in $x$ and non-decreasing
     in $y$.} For all $x, y$ we have
\begin{align}\label{chirule1}
\chi^{\ell} (x,y) + \chi^{-\ell}(y,x) &= 1 + \mathbb{I}[\ell =0]
                                          \mathbb{I}[x=y] 
\end{align}
Moreover,
\begin{align}\label{eq:76}
\chi^\ell (u,y)  \chi^\ell (v,y) =  \chi^\ell ({\max(u,v)},y)  \mbox{ and } 
 \chi^\ell (x,u) \chi^\ell (x,v)  =  \chi^\ell (x,{\min(u,v)}).
\end{align}
\end{lma}

Let $p$ with support $\mathcal{S}(p)$ satisfy Assumption A.  Then for
 any $f \in L^1(p)$ {it is easy to check  from the definition \eqref{eq:53} that
  \begin{align}
\label{eq:32}
    \mathcal{L}^{\ell}_pf(x) &  = \frac{1}{p(x)}\mathbb{E} \left[
   \chi^{\ell}( X, x)\big(f(X)- \mathbb{E}[f(X)]\big)\right]\\ 
    &    = \frac{1}{p(x)}\mathbb{E} \left[
   \big(\chi^{\ell}( X, x) - \mathbb{E}[\chi^{\ell}( X, x)]\big) \big(f(X)- \mathbb{E}[f(X)]\big)\right].\nonumber   
  \end{align}
  }
Next, define
\begin{align}
  \label{eq:29}
  \Phi^{\ell}_p(u, x, v) = \frac{\chi^{\ell}(u, x)\chi^{-\ell}(x, v)}{p(x)}
\end{align}
for all $x \in \mathcal{S}(p)$ and 0 elsewhere.  {This function is
  used in} the following representation formula for the Stein inverse
operator:
 \begin{lma}[Representation formula I]\label{lem:represnt}
   Let $X, X_1, X_2$ be independent copies of $X \sim p$ with support
   $\mathcal{S}(p)$.  Then, for all
   $f \in L^1(p)$ we have
\begin{equation}
  \label{eq:reprsntform}
  -\mathcal{L}_p^{\ell}f(x)  =  \mathbb{E} \left[
       (f(X_2)-f(X_1)) 
\Phi^{\ell}_p(X_1, x,  X_2) 
     \right]. 
\end{equation}
 \end{lma}
\begin{proof}
  The $L^1(p)$ condition on $f$ suffices for the expectation on the
  r.h.s. of \eqref{eq:reprsntform} to be finite for all
  $x \in \mathcal{S}(p)$.  Suppose without loss of generality that
  $\mathbb{E}[f(X)]=0$. 
   Using that $X_1, X_2$ are i.i.d., we reap
{\begin{align*}
& \mathbb{E} \big[
       (f(X_2)-f(X_1)) 
       \chi^{\ell}(X_1, x)  \chi^{-\ell}(x, X_2)
     \big]  \\
& = \mathbb{E} \big[      
       \chi^{\ell}(X_1, x)
     \big]  \mathbb{E}[ f(X_2)\chi^{-\ell}(x, X_2)]-
 \mathbb{E}[ f(X_1)\chi^{\ell}(X_1, x)]                 \mathbb{E} \big[ 
       \chi^{-\ell}(x, X_2)
     \big] \\
&   =  \mathbb{E}\big[\chi^{\ell}(X_1, x)\big]  \mathbb{E} \big[
       f(X_2) 
       \big(1- \chi^{\ell}(X_2, x)\big)
                \big] -
  \mathbb{E} \big[
       f(X_1)
       \chi^{\ell}(X_1, x)
     \big]  \mathbb{E}[\chi^{-\ell}(x, X_2)]\\
  & = - \mathbb{E} \big[
       f(X)
       \chi^{\ell}(X, x)
     \big]  \big(\mathbb{E}[\chi^{\ell}(X,
    x)]+\mathbb{E}[\chi^{-\ell}(x, X)]\big), 
\end{align*}
where in the third line we used the fact that
$\mathbb{E}[f(X) \mathbb{I}[\ell=0] \mathbb{I}[X=x]]=0$ under the
stated assumptions. For the same reasons, we have
$ \mathbb{E} [\chi^\ell(X, x) + \chi^{-\ell}(x, X) ] = 1$ for all
$x\in \mathcal{X}$ and all $\ell \in \left\{ -1, 0, 1 \right\}$.  The
conclusion follows {by} recalling \eqref{eq:32}.}
\end{proof}


The function defined in \eqref{eq:29} allows to perform
``probabilistic integration'' as follows: if
$f \in \mathrm{dom}(\Delta^{-\ell})$ is such that
$(\Delta^{-\ell}f)$ is integrable on $[x_1, x_2] \cap \mathcal{S}(p)$ 
then
\begin{equation}
  \label{eq:28}
f(x_2)-f(x_1) =  \mathbb{E}\left[   \Phi^{\ell}_p(x_1, X, x_2)
  \Delta^{-\ell}f(X)\right] =
\begin{cases}
  \int_{x_1}^{x_2}f'(u) \mathrm{d}u  & (\ell = 0) \\
  \sum_{j=x_1}^{x_2-1} \Delta^+f(j)  & (\ell = -1) \\
  \sum_{j=x_1+1}^{x_2} \Delta^-f(j)  & (\ell = 1) 
\end{cases}
\end{equation}
for all $x_1 < x_2 \in \mathcal{S}(p)$.
If, furthermore,  $f \in
L^1(p)$ then {(by a conditioning argument)}
\begin{equation*}
  \mathbb{E}[(f(X_2) - f(X_1)) \mathbb{I}[X_1<X_2]] = \mathbb{E}\left[
    \Phi^{\ell}_p(X_1, X, X_2) \Delta^{-\ell}f(X)\right]. 
\end{equation*}
{Equation \eqref{eq:28} leads to the}  next representation formula for
the inverse Stein operator. 

\begin{lma}[Representation formula II]\label{lem:can_inverserep2}
  Let $X \sim p$. Define {the kernel $K_p$ on
    $\mathcal{S}(p)\times\mathcal{S}(p)$ by}
\begin{equation}
  \label{eq:15}
  K_p^{\ell}(x, x')  = 
\mathbb{E} [\chi^{\ell}(X, x)
    \chi^{\ell}(X, x') \big]  -
    \mathbb{E} [\chi^{\ell}(X, x)\big] \mathbb{E}\big[\chi^{\ell}(X, x')
    \big]. 
  \end{equation}
  Then $K_p^{\ell}(x, x')$ is symmetric and positive. Moreover, for
  all $f \in \mathrm{dom}(\Delta^{-\ell})$ such that $f \in L^1(p)$
  we have,
  \begin{equation}
\label{eq:24}
    -\mathcal{L}_p^{\ell}f(x)  = \mathbb{E}\left[ \frac{
                                  K_p^{\ell}(X, x)}{p(X)p(x)}
                                  \Delta^{-\ell} f(X) \right].
  \end{equation}
\end{lma}

\begin{proof}
  Symmetry of $K_p^{\ell}$ is immediate. To see that it is positive,
  {applying first \eqref{eq:76} and then \eqref{chirule1},}
\begin{align}
  K_p^{\ell}(x, x') 
  &= \mathbb{E} \left[ \chi^{\ell}(X, \min(x, x'))
  \right] \bigg(1-  \mathbb{E} \left[ \chi^{\ell}(X, \max(x, x'))
  \right] \bigg) \nonumber 
\\ &{=\mathbb{E} \left[ \chi^{\ell}(X, \min(x, x'))
  \right] \mathbb{E} \left[ \chi^{-\ell}(\max(x, x'),X) \right] }   \label{eq:1}
\end{align}
which is necessarily positive.  To prove \eqref{eq:24}, we insert
\eqref{eq:28} into \eqref{eq:reprsntform}, to obtain
\begin{align*}
  -\mathcal{L}_p^{\ell}f(x)  & = \mathbb{E} \left[ \Delta^{-\ell}f(X') \Phi^{\ell}_p(X_1,
                               X', X_2) 
                               \Phi^{\ell}_p(X_1, x, X_2)  \right] \\
  & =  \mathbb{E} \left[ \Delta^{-\ell}f(X')
    \mathbb{E} \left[ \Phi^{\ell}_p(X_1, 
                               X', X_2) 
                               \Phi^{\ell}_p(X_1, x, X_2) \, | \, X{'} \right] \right].
\end{align*}
For all $x, x' \in \mathcal{S}(p)$, by \eqref{eq:76}, 
\begin{align*}
  \mathbb{E} \left[ \Phi^{\ell}_p(X_1, x, X_2) \Phi^{\ell}_p(X_1, x', X_2) \right] 
  &  = \frac{1}{p(x) p(x')} \mathbb{E}[\chi^{\ell}(X, x)\chi^{\ell}(X, x')]
  \mathbb{E}[\chi^{-\ell}(x, X) \chi^{-\ell}(x',X)] 
  \\
  & =\frac{1}{p(x) p(x')} \left(\mathbb{E} \left[ \chi^{\ell}(X, \min(x, x'))
  \right] \mathbb{E} \left[ \chi^{-\ell}(\max(x, x'),X) \right] \right).
\end{align*}
Using \eqref{eq:1}, we 
  recognize the kernel
$K_p^{\ell}(x, x')$ in the numerator, and identity \eqref{eq:24}
follows.
\end{proof}


\begin{exm} Representations \eqref{eq:reprsntform} and \eqref{eq:24}
  can easily be applied to obtain representations for the Stein kernel
  $\tau_p^{\ell}(x)$:
  \begin{align*}
    \tau_p^{\ell}(x)&  = -\mathcal{L}_p^{\ell}(\mathrm{Id})(x)  
                      = \mathbb{E} \left[ (X_2-X_1) \Phi^{\ell}_p(X_1, x,  X_2) \right]
                      =\mathbb{E}\left[ \frac{K_p^{\ell}(X,
                      x)}{p(X)p(x)} \right]. 
  \end{align*}
  In particular the Stein kernel is positive {on}
  $\mathcal{S}(p)$.
\end{exm}

Identity \eqref{eq:reprsntform} seems to be new, although it is
present in non-explicit form in \cite[Equation (4.16)]{CS11}.
Representation \eqref{eq:24} is, in the continuous $\ell = 0$ case,
already available in \cite{saumard2018weighted}. The kernel
$K_p^{\ell}(x, x')$ is a classical object in the theory of covariance
representations and inequalities; an early appearance is attributed by
\cite{rao2006matrix} to \cite{hoffding1940masstabinvariante} (see
\cite[pp 57--109]{hoeffding2012collected} for an English translation).
The {perhaps not very surprising} extension to the discrete
{case} is, to the best of our knowledge, new.  

\medskip

As a first {result from our set-up}, 
\eqref{eq:24} applied to the function
$f(x) = \mathcal{T}_p^{\ell}1(x)$ immediately gives the following:
\begin{prop}[Menz-Otto formula]\label{prop:ottomenz1}
  Suppose that the constant function $1$ belongs to
  $\mathcal{F}_\ell^{(1)}(p)$, that
  $-\Delta^{-\ell}\mathcal{T}_p^{\ell}1(x)>0$ for almost all
  $x \in \mathcal{S}(p)$ and $\Delta^{-\ell}(\mathcal{T}_p^{\ell}1)
  \in L^1(\mu)$. 
  Then, for every $x'\in \mathcal{S}(p)$, the function
  \begin{equation}
    \label{eq:96}
    p^{\star}_{x'}(x) = \frac{K_p^{\ell}(x, x')}{p(x')}
    \bigg(-\Delta^{-\ell}\mathcal{T}_p^{\ell}1(x)\bigg) 
  \end{equation}
  is a density on $\mathcal{S}(p)$ {with respect to $\mu$}. 
\end{prop}

\begin{proof} 
{From \eqref{eq:24} with $f(x) = \mathcal{T}_p^{\ell}1(x) = \frac{\Delta^\ell(p(x))}{p(x)}$ and $ \mathcal{L}_p^{\ell}f(x)  =1$, 
\begin{align*}  
  1 &= - \mathbb{E}\left[ \frac{K_p^{\ell}(X, x)}{p(X)p(x)}
                                  \Delta^{-\ell}  \mathcal{T}_p^{\ell}1(X) \right] 
    = 
    \int_a^b  \frac{K_p^{\ell}(u, x)}{p(u)p(x)}
    \left( {-} \Delta^{-\ell}  \mathcal{T}_p^{\ell}1(u) \right) p(u) \mu(du) \\
    &= 
    \int_a^b  \frac{K_p^{\ell}(u, x)}{p(x)}
    \left( {-} \Delta^{-\ell}  \mathcal{T}_p^{\ell}1(u) \right)  \mu(du)
    = \int_a^b  p^{\star}_{x'}(x) \mu(dx).
                                  \end{align*} 
                       Since $0 \le K_p^{\ell}(x, x') \le 1$,
                       {the integral exists because
                         $-\Delta^{-\ell}\mathcal{T}_p^{\ell}1(x) 
   \frac{K_p^{\ell}(x, x)}{p(x)} \in
  L^1(p)$}
                                   and,} by
                                assumption,
                                $
                                -\Delta^{-\ell}\mathcal{T}_p^{\ell}1(x)
                                > 0$. Hence the assertion follows.
\end{proof}

\begin{rmk} If $ \frac{K_p^{\ell}(x, x)}{p(x)}$ is bounded, then
    the assumptions in Proposition \ref{prop:ottomenz1} are satisfied
    as soon as $-\Delta^{-\ell}\mathcal{T}_p^{\ell}1 \in L^1({\mu})$. The
    proposition thus applies when $\ell = 0$ and $p(x) = e^{-H(x)}$
    with $H$ a strictly convex function {such that
      $\lim_{x\to\pm \infty} H'(x) = 0$}. This puts us in the context
    studied by \cite{menz2013uniform} and formula \eqref{eq:96} is
    equivalent to their \cite[Equation (14)]{menz2013uniform}; we
    return to this in Section
    \ref{sec:variance-expansion}. 
\end{rmk}

\medskip

{The next proposition gives some properties of $K_p^{\ell}(x,x')$.} 

  \begin{prop}
    (i) {It holds that for all $x, x' \in \mathcal{S}(p),$ that}
    $K_p^{\ell}(x, x')\le{K_p^{\ell}( \min(x, x'), \min(x,
      x'))}$. (ii) If $\mathbb{E}[\chi^{\ell}(X, x)]/p(x)$ is non
    decreasing, {then the function
      $x \mapsto {K_p^{\ell}(x, x')}/{p(x)}$ is non-decreasing for
      $x < x'$. (iii) If $\mathbb{E}[\chi^{-\ell}(x,X)]/p(x)$ is non
      increasing, then the function
      $x \mapsto {K_p^{\ell}(x, x')}/{p(x)}$ is non-increasing for
      $x > x'$. }
  

  \end{prop}
  
  \begin{proof}
{To see (i), we start from \eqref{eq:1}, 
$  K_p^{\ell}(x, x') 
   {=\mathbb{E} \left[ \chi^{\ell}(X, \min(x, x'))
  \right] \mathbb{E} \left[ \chi^{-\ell}(\max(x, x'),X) \right] }  
$
and by Lemma \ref{lem:chirules}, $\chi^{-\ell}(u,x)$ is non-increasing
in $u$: $\chi^{-\ell}(\max(x, x'),X) \le \chi^{-\ell}(\min(x,
x'),X)$. We deduce that
\begin{equation*}
  K_p^{\ell}(x, x')   \le
    {\mathbb{E} \left[ \chi^{\ell}(X, \min(x, x'))
  \right] \mathbb{E} \left[ \chi^{-\ell}(\min(x, x'),X) \right] }
\end{equation*}
Assertion (i) follows by reverting the argument, because
$ (\chi^{\ell}(X, \min(x, x'))^2 = \chi^{\ell}(X, \min(x, x')) $.}
To see (ii), assume that  $\mathbb{E}[\chi^{\ell}(X, x)]/p(x)$ is non-decreasing. Then with \eqref{eq:1}, for $x < x'$, 
 \begin{align*}
 \frac{1}{p(x)}  K_p^{\ell}(x, x') 
   &= \frac{1}{p(x)} \mathbb{E} \left[ \chi^{\ell}(X, \min(x, x'))
  \right] \mathbb{E} \left[ \chi^{-\ell}(\max(x, x'),X) \right] \\
  &= \left( \frac{1}{p(x)} \mathbb{E} \left[ \chi^{\ell}(X, x)
  \right] \right) \mathbb{E} \left[ \chi^{-\ell}(x',X) \right];
\end{align*}
the second factor is a constant, and the first factor is assumed to be
non-decreasing. Hence the assertion follows. For (iii), assume that
$\mathbb{E}[\chi^{-\ell}(x,X)]/p(x)$ is non-increasing; then similarly
as above, for $x > x'$,
 \begin{align*}
 \frac{1}{p(x)}  K_p^{\ell}(x, x') 
   &= \frac{1}{p(x)} \mathbb{E} \left[ \chi^{\ell}(X, \min(x, x'))
  \right] \mathbb{E} \left[ \chi^{-\ell}(\max(x, x'),X) \right] \\
  &= \left( \mathbb{E} \left[ \chi^{\ell}(X, x')
  \right] \right) \left(  \frac{1}{p(x)}\mathbb{E} \left[
    \chi^{-\ell}(x,X) \right] \right); 
\end{align*}
the first factor is constant, and the second factor is non
increasing. Hence the assertion follows.
    \end{proof}
    Figures \ref{fig:figKPP} and \ref{fig:figKPP2} {display the
      functions {$x \mapsto {K_p^{\ell}(x, x')}/{p(x)}$ (for various
        values of $x'$) and} $x \mapsto {K_p^{\ell}(x, x)}/{p(x)}$}
    for the standard normal and several choices of the parameters in
    beta, gamma, binomial, Poisson and hypergeometric distributions.
    \begin{exm}  The following facts are easy to prove:
  \begin{enumerate}
  \item If $p(x)$ is the standard normal distribution then
    ${K_p^{0}(x, x)}/{p(x)}$ behaves as $1/|x|$ for large $|x|$, see
    Figure \ref{fig:sfig11b}.
  \item If $p(x)$ is gamma then ${K_p^{0}(x, x)}/{p(x)}$ behaves as a
    constant for large $x$, see Figure \ref{fig:sfig31b}.
 
  \item The function $x \mapsto {K_p^{0}(x, x)}/{p(x)}$ is not in
    $L^1(p)$ for $p$ a Cauchy distribution.
    \item If $p$ is strictly-log concave then ${K_p^{\ell}(x, x)}/{p(x)}$ is bounded.
    \end{enumerate}
    
\end{exm}

 \begin{figure}
\begin{subfigure}{.49\textwidth}
  \centering
  \includegraphics[width=1\linewidth]{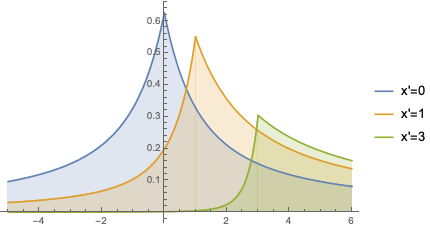}
  \caption{}
  \label{fig:sfig11}
\end{subfigure}%
\begin{subfigure}{.49\textwidth}
  \centering
  \includegraphics[width=1\linewidth]{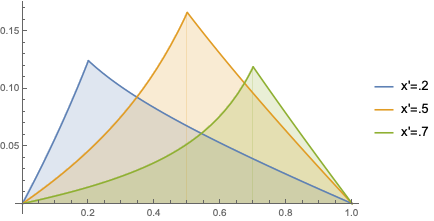}
  \caption{}
  \label{fig:sfig21}
\end{subfigure}
\begin{subfigure}{.49\textwidth}
  \centering
  \includegraphics[width=1\linewidth]{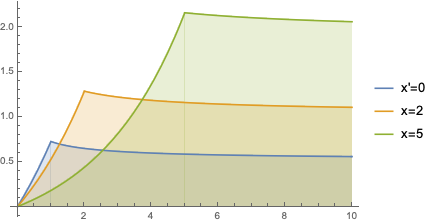}
  \caption{}
  \label{fig:sfig31}
\end{subfigure}
\begin{subfigure}{.49\textwidth}
  \centering
  \includegraphics[width=1\linewidth]{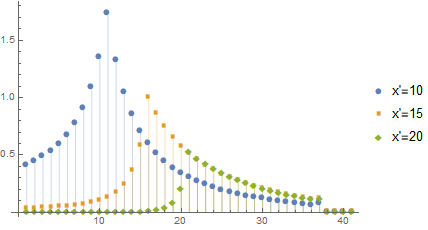}
  \caption{}
  \label{fig:sfigbin1}
\end{subfigure}%

\begin{subfigure}{.49\textwidth}
  \centering
  \includegraphics[width=1\linewidth]{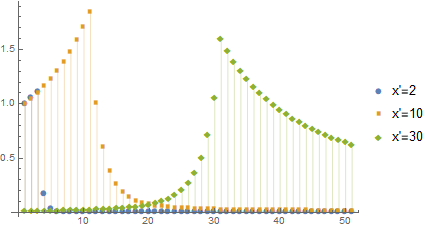}
  \caption{}
  \label{fig:sfigpoi1}
\end{subfigure}
\begin{subfigure}{.49\textwidth}
  \centering
  \includegraphics[width=1\linewidth]{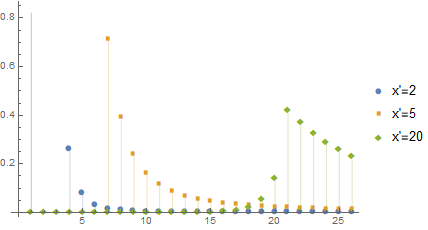}
  \caption{}
  \label{fig:sfighyp1}
\end{subfigure}
\caption{{\it The functions
    $x \mapsto {K_p^{\ell}(x, x')}/{p(x)}$ for different (fixed) values of $x'$ and
    $p$ the standard normal distribution (Figure~\ref{fig:sfig11});
    beta distribution with parameters 1.3 and 2.4
    (Figure~\ref{fig:sfig21}); gamma distribution with parameters 1.3
    and 2.4 (Figure~\ref{fig:sfig31}); binomial distribution with
    parameters $(50, 0.2)$ (Figure~\ref{fig:sfigbin1}); Poisson
    distribution with parameter 20 (Figure~\ref{fig:sfigpoi1});
    hypergeometric distribution with parameters 100, 50 and 500
    (Figure~\ref{fig:sfighyp1})}.
\label{fig:figKPP}}
\end{figure}

 \begin{figure}
\begin{subfigure}{.49\textwidth}
  \centering
  \includegraphics[width=1\linewidth]{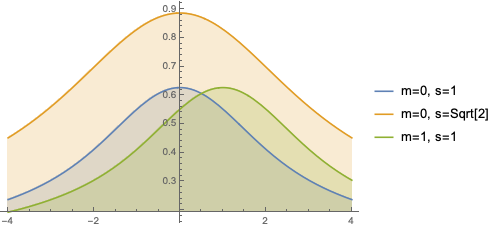}
  \caption{}
  \label{fig:sfig11b}
\end{subfigure}%
\begin{subfigure}{.49\textwidth}
  \centering
  \includegraphics[width=1\linewidth]{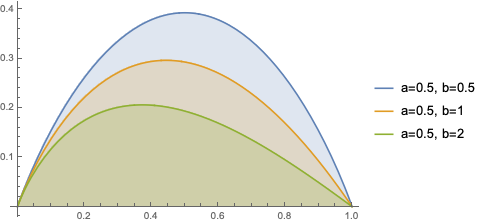}
  \caption{}
  \label{fig:sfig21b}
\end{subfigure}

\begin{subfigure}{.49\textwidth}
  \centering
  \includegraphics[width=1\linewidth]{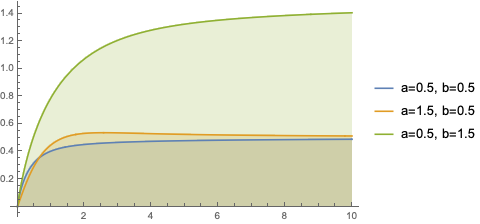}
  \caption{}
  \label{fig:sfig31b}
\end{subfigure}
\begin{subfigure}{.49\textwidth}
  \centering
  \includegraphics[width=1\linewidth]{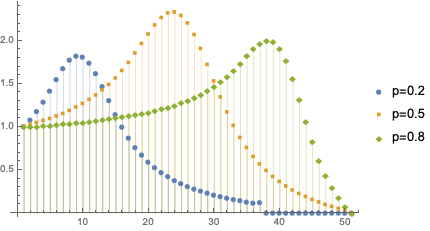}
  \caption{}
  \label{fig:sfigbin1b}
\end{subfigure}%

\begin{subfigure}{.49\textwidth}
  \centering
  \includegraphics[width=1\linewidth]{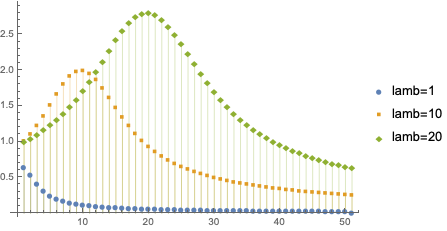}
  \caption{}
  \label{fig:sfigpoi1b}
\end{subfigure}
\begin{subfigure}{.49\textwidth}
  \centering
  \includegraphics[width=1\linewidth]{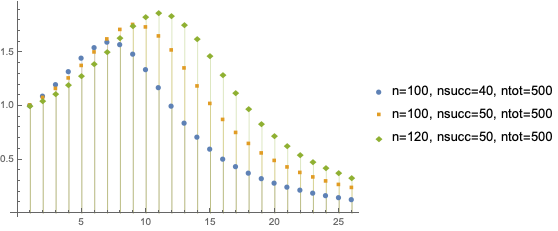}
  \caption{}
  \label{fig:sfighyp1b}
\end{subfigure}
\caption{{\it The functions $x \mapsto {K_p^{\ell}(x, x)}/{p(x)} $ for
  different parameter values $p$ the standard normal distribution
  (Figure~\ref{fig:sfig11b}); beta distribution
  (Figure~\ref{fig:sfig21b}); gamma distribution
  (Figure~\ref{fig:sfig31b}); binomial distribution with parameters
  (Figure~\ref{fig:sfigbin1b}); Poisson distribution
  (Figure~\ref{fig:sfigpoi1b}); hypergeometric distribution
  (Figure~\ref{fig:sfighyp1b})}.
\label{fig:figKPP2}}
\end{figure}

\subsection{Sufficient conditions and integrability}
\label{sec:suff-cond}

As anticipated, we now study the conditions under which the IBP Lemmas
\ref{lma:skadj} and \ref{lem:stIBP1} hold. {All proofs are technical
manipulations of basic calculus and relegated to the Appendix \ref{sec:some-non-essential}.}

We
start by the decryption of the conditions for Lemma
\ref{lma:skadj}. Recall the notations $a_{\ell}$ and $b_{\ell}$ from
\eqref{eq:31}. {Furthermore if $\ell =0$ we write
  $f(a^+) = \lim_{x \to a, x>a}f(x)$ and
  $f(b^-) = \lim_{x \to b, x<b}f(x)$.  In the case that $a = -\infty$
  or $b = \infty$, for $\ell \in \{-1, 0, 1\}$, we write
  $f(-\infty^+) = \lim_{x \to - \infty}f(x)$ and
  $f(\infty^-) = \lim_{x \to \infty}f(x)$. To simplify notation, if
  $\ell \in \{-1, 1\}$ and $a \ne -\infty$, we write $f(a^+) =f(a)$,
  and similarly, if $b \ne \infty, f(b^-) = f(b)$.  }

\begin{prop}[Sufficient conditions for IBP -- version
  1] \label{prop:suff1}Let $ f\in \mathrm{dom}(p, \Delta^{\ell})$ and
  $g \in \mathrm{dom}(\Delta^{-\ell})$.  In order for \eqref{eq:3} to
  hold it suffices that they jointly satisfy the following conditions
  \begin{align}
    \label{eq:11}
  &   (\mathcal{T}_p^{\ell} f) g  \mbox{ and } f (\Delta^{-\ell}g)  \in L^1(p)\\
   &  \label{eq:9}
{
    f(b^-+a_{\ell}) g(b^-+a_{\ell}-\ell) p(b^-+a_{\ell}) =  f(a^+-b_{\ell})g(a^+-b_{\ell}-\ell) p(a^+-b_{\ell}).
     }
  \end{align}

\end{prop}

For ease of future reference, we spell out \eqref{eq:9} in the three
cases that interest us:
\begin{align*}
  \begin{cases}
    f(b^-) g(b^-)p(b^-) =    f(a^+) g(a^+)p(a^+)  & \ell = 0\\
  f(b^-) g(b^-+1) p(b^-) = 0 & \ell = -1\\
f(a^+)g(a^+-1) p(a^+) =0 & \ell = 1.
  \end{cases}                                                         
\end{align*}

{ We now derive a set of (almost) necessary and sufficient conditions
  under which \eqref{eq:21} holds. }
 
\begin{prop}[{Sufficient} conditions for IBP -- version 2]\label{prop:ibp2}
  Let $g \in \mathrm{dom}(\Delta^{-\ell})$. In order for \eqref{eq:21}
  to hold, it is necessary and sufficient that they jointly satisfy
  the three following conditions:
     \begin{align}
       \label{eq:91}
& {f},  g  \mbox{ and } f g \in L^1(p),  \\
& \label{eq:18}\mathcal{L}_p^{\ell} f ( \Delta^{-\ell}g)  \in L^1(p)\\
&
     \big(\mathcal{L}_p^{\ell}f(b^-+ a_{\ell})\big) g(b^-+ a_{\ell}-\ell)p(b^-+ a_{\ell}) =
\big( \mathcal{L}_p^{\ell}f(a^+- b_{\ell})\big) g(a^+- b_{\ell}-\ell)p(a^+- b_{\ell})   . \label{eq:92} 
     \end{align}
\end{prop}

Requirement \eqref{eq:91} is natural and condition \eqref{eq:92} is
mild as it is satisfied as soon as $g$ and/or $f$ are well behaved at
the edges of the support. {Condition \eqref{eq:18} (which is already
stated in the original statement of Lemma \ref{lem:stIBP1}) is harder
to fathom.  In order to make it even more readable, and facilitate the
connexion with the literature, we {specialise} the conditions further
in our next result.}

  \begin{prop}\label{lma:suffcond1} 
    Let $f, g$ and $fg \in L^1(p)$. If
    $g \in \mathrm{dom}(\Delta^{-\ell})$ is of bounded variation and
    satisfies the following two conditions: 
    \begin{enumerate}
    \item \label{item:cond1g} {$g(a^+ -b_\ell -\ell) \mathbb{P}( X
        \le a^+ - a_\ell- b_{\ell}) = 0$ and $g(b^- +
      a_\ell-\ell) \mathbb{P}(X \ge b^- +a_\ell + b_\ell) = 0$} 
\item \label{item:cond2g} {$g(a^+ -b_\ell-\ell ) \mathbb{E}[|f(X)|
    \chi^{\ell}(X \le a^+ - b_\ell)] = 0$ and $g(b^- +
      a_\ell - \ell)\mathbb{E}[|f(X)|
    \chi^{-\ell}(b^- + a_\ell, X)]  = 0$},
    \end{enumerate}
    then \eqref{eq:92} holds. In particular if $f$ is bounded or in
    $L^2(p)$, then the condition \ref{item:cond2g} above is implied by
    condition \ref{item:cond1g}.

\end{prop}

\begin{rmk}

  
  This assumption is closer to what is to be found in the literature,
  see e.g.\ \cite{saumard2018weighted} in the case $\ell=0$. The main
  difference between the classical assumptions and ours is that we
  only impose conditions on one of the functions. We stress that there
  is a certain degree of redundancy in the items \ref{item:cond1g} and
  \ref{item:cond2g} together with the assumption that $g \in L^1(p)$
  and is of bounded variation; the statement could be shortened at the
  loss of readability.

\end{rmk}

In the sequel, to preserve as much generality as possible and not
overburden the statements, we will simply require that ``the
assumptions of Lemma \ref{lem:stIBP1} are
satisfied.''

 \subsection{The inverse Stein operator} \label{subsec:inverseOp}

 We conclude this section by exploring easy consequences of the
 representations from Section \ref{sec:three-repr-inverse}. These
 results are also of independent interest to practitioners of Stein's
 method.
 
{
 \begin{lma}
If   $f, \mathcal{L}_p^{\ell}f(X) \in L^1(p)$ then
\begin{equation}\label{eq:4}
    \mathbb{E} \left[ -\mathcal{L}_p^{\ell}f(X) \right] = \mathbb{E}
    \left[ (X_2-X_1)^+(f(X_2)-f(X_1)) \right]                                                              =      \frac{1}{2}        \mathbb{E}
    \left[ (X_2-X_1)(f(X_2)-f(X_1)) \right]                           
\end{equation}
where $(\cdot)^+$ denotes the positive part of $(\cdot)$.  In
particular, if the conditions of Lemma \ref{lem:stIBP1} are satisfied
with $f(x) = g(x) = \mathrm{Id}$, then
$ \mathbb{E}[\tau_p^{\ell}(X)] = \mathrm{Var}(X).$
\end{lma}


\begin{proof}
Representation \eqref{eq:reprsntform} gives
    $\mathbb{E}\big[ -\mathcal{L}_p^{\ell}f(X)\big]  = \mathbb{E} [
                                             ( f(X_2)-f(X_1) )
                                              \Phi_p^{\ell}(X_1, X,
                                              X_2)].$
    Using \eqref{eq:28} with $f(x) = x$, we have
    $\mathbb{E}[\Phi_p^{\ell}(x_1, X, x_2)] = (x_2 - x_1)^+$. Hence, after
    conditioning with respect to $X_1, X_2$, the first equality in \eqref{eq:4}
    follows. The second equality  follows by symmetry. The second
    claim is immediate under the stated assumptions. 
\end{proof}
\begin{rmk}
  Once again, our assumptions are minimal but not transparent. It is
  easy to spell out these conditions explicitly for any specific
  target. For instance if $X$ has bounded support or support $\R$ then
  finite variance suffices.
\end{rmk}
  }
  \begin{prop}\label{prop:applirep} {Suppose that all test functions
      satisfy the conditions in Lemma \ref{lem:stIBP1}.  Let
      $ \|f\|_{{\mathcal{S}}(p), \infty} = \sup_{x \in
        {\mathcal{S}}(p)} | f(x)|.$ }
\begin{enumerate}
\item If $f$ is monotone then $\mathcal{L}_p^{\ell}f(x)$ does not
  change sign. 
\item (Uniform bounds Stein bounds) Consider, for $h$ and $\eta$ in
  $L^1(p)$ the function
  \begin{align*}
     g_h^{p, \ell, \eta} (x) 
 = \frac{\mathcal{L}_p^{\ell} h(x+ \ell)}{\mathcal{L}_p^{\ell} 
 \eta(x+\ell)}
  \end{align*}
   defined in \eqref{eq:8} {which
    solves the $\eta$-Stein equation \eqref{eq:43} for $h$}. 
  If {$\eta$ is monotone  and }
  $|h(x)-h(y)| \le k|\eta(x)-\eta(y)|$ for all {$x, y \in \mathcal{S}(p)$, then} 
  \begin{align*}
    \|g_h^{p, \ell, \eta} \|_{{\mathcal{S}}(p), \infty} \le k. 
  \end{align*}
  In particular, if $h \in L^1(p)$ is Lipschitz continuous with
  Lipschitz constant $1$ then the above applies with $ \eta(x) = x$,
  and $\|g_h^{p, 0, \mathrm{Id}} \|_{{\mathcal{S}}(p), \infty} \le 1$.
\item (Non uniform bounds Stein bounds) 
\begin{align}
  \label{eq:47}
  \left| \mathcal{L}_p^{\ell}f(x)  \right| \le 2{ \|f\|_{{\mathcal{S}}(p), \infty} }  
  \frac{\mathbb{E}[\chi^{\ell}(X_1, x)] \mathbb{E}[\chi^{-\ell}(x,
    X_2)]}{p(x)} 
\end{align}
for all $x \in \mathcal{S}(p)$. 
\end{enumerate}

\end{prop}
\begin{proof}
  Recall representation
  \eqref{eq:reprsntform} which states that
  \begin{equation*}
  -\mathcal{L}_p^{\ell}f(x)  =  \frac{1}{p(x)}\mathbb{E} \left[
       (f(X_2)-f(X_1)) 
  \chi^{\ell}(X_1, x)\chi^{-\ell}(x, X_2)\right]. 
\end{equation*}
\begin{enumerate}
\item If $f(x)$ is monotone then $f(X_2) - f(X_1)$ is of constant sign
  conditionally on  $\chi^{\ell}(X_1, x)\chi^{-\ell}(x, X_2) = {\mathbb{I} [ X_1 + a_\ell \le x \le X_2 - b_{\ell}]=1}$, 
  {because on this event, $X_1 \le X_2 - \mathbb{I} [\ell \ne 0]$.} Hence the first assertion follows. 
\item Suppose that the function $\eta$ is strictly decreasing.  By
  definition of $g$ we have, under the stated conditions,
  \begin{align*}
    |g(x)| &  = \left| \frac{-\mathbb{E} \left[
       (h(X_2)-h(X_1)) 
  \chi^{\ell}(X_1, x{+\ell})\chi^{-\ell}(x{+\ell}, X_2)\right]}{-\mathbb{E} \left[
       ( \eta(X_2)-\eta(X_1)) 
             \chi^{\ell}(X_1, x{+\ell})\chi^{-\ell}(x{+\ell}, X_2)\right] } \right|
    \\
    &  =  \frac{\left|-\mathbb{E} \left[
       (h(X_2)-h(X_1)) 
  \chi^{\ell}(X_1, x{+\ell})\chi^{-\ell}(x{+\ell}, X_2)\right]\right|}{\mathbb{E} \left[
       ( \eta(X_1)-\eta(X_2)) 
             \chi^{\ell}(X_1, x{+\ell})\chi^{-\ell}(x{+\ell}, X_2)\right] } \\
    & \le  \frac{\mathbb{E} \left[
       |h(X_2)-h(X_1)| 
  \chi^{\ell}(X_1, x{+\ell})\chi^{-\ell}(x{+\ell}, X_2)\right]}{\mathbb{E} \left[
       | \eta(X_2)-\eta(X_1)|
             \chi^{\ell}(X_1, x{+\ell})\chi^{-\ell}(x{+\ell}, X_2)\right] } \\
    & \le k
  \end{align*}
  for $x \in \mathcal{S}(p)$. 
\item By   \eqref{eq:reprsntform},
  \begin{align*}
\big|  \mathcal{L}_p^{\ell}f(x)\big|  {p(x)} &= \mathbb{E} \left[
       |f(X_2)-f(X_1)|
  \chi^{\ell}(X_1, x)\chi^{-\ell}(x, X_2)\right]\\
  &  \le 2 \|f\|_{{\mathcal{S}}(p), \infty}  \mathbb{E} \left[
    \chi^{\ell}(X_1, x)\chi^{-\ell}(x, X_2)\right]\\
    &\le 2 \|f\|_{{\mathcal{S}}(p), \infty}  \mathbb{E} \left[
    \chi^{\ell}(X_1, x) \right]\mathbb{E} \left[\chi^{-\ell}(x, X_2)\right]
  \end{align*}
  which leads to the conclusion. 
\end{enumerate}
\end{proof}

\begin{exm}\label{ex:gaussmills}
  If $p = \phi$ is the standard Gaussian with cdf $\Phi$, then
  $\ell =0$ and the third bound in Proposition \ref{prop:applirep}
  reduces to
  ${2} || f ||_\infty \frac{\Phi(x)( 1 - \Phi(x))}{\phi(x)}$. The
  ratio ${\Phi(x)\big(1-\Phi(x)\big)}/{\phi(x)}$ is closely related to
  Mill's ratio of the standard normal law.  The study of such a
  function is classical and much is known. For instance, we can apply
  \cite[Theorem 2.3]{baricz2008mills} to get
\begin{equation}\label{eq:5}
  \frac{1}{\sqrt{x^2+4}+x} \le
  \frac{\Phi(x)\big(1-\Phi(x)\big)}{\phi(x)} 
  \le \frac{4}{\sqrt{x^2+8}+3x} 
\end{equation}
for all $x \ge 0$.   Moreover, 
${\Phi(x)\big(1-\Phi(x)\big)}/{\phi(x)}  \le {\Phi(0)\big(1-\Phi(0)\big)}/{\phi(0)} =
1/2\sqrt{\pi/2} \approx 0.626$. In particular 
{Proposition \ref{prop:applirep} recovers the well-known bound}
$\| 
\mathcal{L}_{p}^{\ell}f \|_{\infty} \le \sqrt{\pi/2} \|f\|_{\infty}$,
see e.g.\ \cite[Theorem 3.3.1]{NP11}.
\end{exm}

\section{Covariance identities and inequalities}
\label{sec:variance-expansion}


We start with an easy lower bound inequality, which follows
immediately from Lemma~\ref{lma:skadj}.
\begin{prop}[Cramer-Rao type bound] \label{prop:cramerao} Let
  $g \in L^2(p)$.  For any $f \in \mathcal{F}_{\ell}^{(1)}(p)$ such
  that $\mathcal{T}_p^{\ell}f \in L^2(p)$ and the assumptions of
  Lemma~\ref{lma:skadj} are satisfied:
 \begin{equation}
\label{eq:13}
 \mathrm{Var}[g(X)]  \ge \frac{\mathbb{E}\left[ f(X) (\Delta^{-\ell}g(X))\right]^{2}}{\mathbb{E} \left[
\big(\mathcal{T}_p^{\ell}f(X)\big)^2 \right]}
  \end{equation}
  with equality if and only if there exist $\alpha, \beta$ real
  numbers such that 
  $g(x) = \alpha \mathcal{T}_p^{\ell}f(x) + \beta$ for all
  $x \in \mathcal{S}(p)$.    
\end{prop}
\begin{proof}
  The lower bound \eqref{eq:13} follows from the fact that
  $\mathcal{T}_p^{\ell} f(X) \in \mathcal{F}^{(0)}(p)$ for all
  $f \in \mathcal{F}^{{(1)}}_{\ell}(p)$ {by Lemma
    \ref{def:can_class}}.
  Therefore, from \eqref{eq:3}, we have
\begin{align*}
\left\{  \mathbb{E} \left[  f(X) \big(\Delta^{-\ell} g(X) \big) \right]\right\}^2 & =
\left\{   \mathbb{E}
 \left[
 \big(\mathcal{T}_p^\ell
 f(X)
  \big)
 g(X)
\right] \right\} ^2  
 = \left\{  \mathbb{E} 
  \left[
  \big(\mathcal{T}_p^\ell
   f(X)
 \big)
(g(X) - \mathbb{E}  [g(X)]) 
 \right]\right\}^2 \\
& \le \mathbb{E} \left[
\big(\mathcal{T}_p^{\ell}f(X)\big)^2 \right] \mathrm{Var}[g(X)]
\end{align*}
by the Cauchy-Schwarz inequality. 
\end{proof}

Upper bounds require some more work. We start with an easy
consequence of our framework.

\begin{cor}[First order covariance
  identities]\label{prop:stein-covar-ident} 
  For all $f, g$ that jointly satisfy the assumptions of
  Lemma~\ref{lem:stIBP1}, we have
\begin{align}
   \mathrm{Cov}[f(X), g(X)] 
& = \mathbb{E} \left[ \Delta^{-\ell} f(X)
     \frac{ K_p^{\ell}(X, X')}{p(X)p(X')}\Delta^{-\ell}
                                g(X')\right]. \label{eq:14}
\end{align}
{Moreover, if choice $f = \mathrm{Id}$ is allowed, then 
 \begin{equation}
    \label{eq:25b}
    \mathrm{Cov}[X, g(X)] =  \mathbb{E} \left[ \tau_p^\ell(X) \Delta^{-\ell}g(X) 
    \right].  
  \end{equation}
}
\end{cor}

\begin{rmk}\label{rem:hoeffding}
  Identity \eqref{eq:14} is provided in \cite{menz2013uniform} (see
  their equation (11)) in the case $\ell = 0$ for a log-concave
  density. Some of the history of this identity, including the
  connection with a classical identity of Hoeffding
  \cite{hoffding1940masstabinvariante}, is provided in \cite[Section
  2]{saumard2018efron}. The earliest version of the same identity
  (still for $\ell = 0$) we have found in
  \cite{cuadras2002covariance}, along with applications to measures of
  correlation as well as further references. A similar identity is
  provided in \cite{menz2013uniform}, without explicit conditions; a
  clear statement is given in \cite[Corollary 2.2]{saumard2018efron}
  where the identity is proved for absolutely continuous $f \in L^r$
  and $g \in L^s$ with conjugate exponents. Our approach shows that it
  suffices to impose regularity on one of the functions for the
  identity to hold. 
\end{rmk}

\begin{proof}
  Let $\bar{f}(x) = f(x) - \mathbb{E}[f(X)]$. {{Note that
      $\Delta^{\ell} {\bar{f}} = \Delta^\ell f.$}} To obtain
  \eqref{eq:14} we start from \eqref{eq:21} and note that if $f, g$
  satisfy the assumptions of Lemma \ref{lem:stIBP1}, then
   \begin{equation}
    \label{eq:25}
    \mathrm{Cov}[f(X), g(X)] =  \mathbb{E} \left[ - {{ \{}}
      \mathcal{L}_p^{\ell}f(X) {{ \}}}\Delta^{-\ell}g(X) 
    \right].  
  \end{equation}
  From this equation, \eqref{eq:25b} follows immediately. 
  Applying \eqref{eq:24} we obtain
\begin{align*}
  \mathrm{Cov}[f(X), g(X)]  & = \mathbb{E}
                              \bigg[ \mathbb{E}\bigg[\frac{K_p(X',
                              X)}{p(X')p(X)}\Delta^{-\ell}f(X') \, |
                              \, X\bigg]\Delta^{-\ell}g(X)
                              \bigg]   
\end{align*}
which gives the claim after removing the conditioning. 
\end{proof}

\begin{exm} {Example \ref{ex:binom2} and identity \eqref{eq:25b}
    give the following covariance identities.}
\begin{itemize}
\item \label{ex:binom4} Binomial distribution:  
 For all functions $g : \Z \to \R$ that are
bounded on $[0, n]$, 
\begin{equation*}
\mathrm{Cov}[X, g(X)] = \mathbb{E} \left[ (1-{\theta})X
    \Delta^{-} g(X) \right] = {\theta}  \mathbb{E} \left[ (n-X)
    \Delta^+g(X)\right] .
\end{equation*}
Combining the two identities we also arrive at
\begin{equation*} 
\mathrm{Cov}[X, g(X)]= \mathrm{Var}[X]
  \mathbb{E}[\nabla_{\mathrm{bin}(n, {\theta})}g(X)]
\end{equation*}
with $\nabla_{\mathrm{bin}(n, {\theta})}$ the ``natural'' binomial
gradient
$\nabla_{\mathrm{bin}(n, {\theta})} g(x) = \frac{x}{n} \Delta^-g(x) +
\frac{n-x}{n} \Delta^+g(x)$ from \cite{hillion2011natural}.
\item\label{ex:beta2} Beta distribution: 
 For all absolutely continuous $g$ such that
$ \mathbb{E} \left[ | X (1-X) g'(X)| \right] <\infty$,
\begin{equation*}
\mathrm{Cov}[X, g(X)] = \frac{1}{\alpha+\beta}  \mathbb{E} \left[ X (1-X)
    g'(X) \right] .
\end{equation*}
\end{itemize} 
\end{exm}

It is of interest to work as in \cite{K85} to obtain a corresponding
upper bound, which would provide some ``weighted Poincar\'e
inequality'' such as those described in \cite{saumard2018weighted}.
The representation formulae \eqref{eq:14} turns out to simplify the
work considerably.

\begin{thm}\label{thm:klaassenbounds}  Fix  $h \in L^1(p)$ a
  decreasing function. For all $f, g$ which     $\mathcal{T}_p^{\ell}c \in L^2(p)$satisfy the assumptions of
  Lemma \ref{lem:stIBP1} we have
  \begin{equation}
    \label{eq:KlaassenCov}
    \left|   \mathrm{Cov}[f(X), g(X)]  \right| \le \sqrt{\mathbb{E}\left[ (\Delta^{-\ell}f(X))^2
        \frac{-\mathcal{L}_p^{\ell}h(X)}{\Delta^{-\ell}h(X)} \right]}\sqrt{\mathbb{E}\left[ (\Delta^{-\ell}g(X))^2
        \frac{-\mathcal{L}_p^{\ell}h(X)}{\Delta^{-\ell}h(X)} \right]} 
  \end{equation}
  with equality if and only if there exist $\alpha_i, i=1, \ldots, 4$
  real numbers such that $f(x) = \alpha_{1} h(x) + \alpha_2$ and
  $g(x) = \alpha_{3} h(x) + \alpha_4$ for all $x \in
  \mathcal{S}(p)$. 
\end{thm}
 
\begin{proof}
  We simply apply \eqref{eq:14} and the Cauchy-Schwarz inequality to
  obtain
\begin{align*}
 &  \left| \mathrm{Cov}[f(X), g(X)] \right|  = 
\left|        \mathbb{E} \left[\Delta^{-\ell} f(X)   \frac{K_p^{\ell}(X, X')}{p(X)p(X')} \Delta^{-\ell} g(X')
       \right] \right|  \\
  & =   \left| \mathbb{E} \left[ \left\{ \frac{\Delta^{-\ell}
    f(X)}{\sqrt{-\Delta^{-\ell} h(X)}}  \sqrt{-\frac{K_p^{\ell}(X,
    X')}{p(X)p(X')} \Delta^{-\ell}h(X')}  \right\}   \left\{ \frac{\Delta^{-\ell}
    g(X')}{\sqrt{-\Delta^{-\ell} h(X')}}  \sqrt{-\frac{K_p^{\ell}(X,
    X')}{p(X)p(X')} \Delta^{-\ell}h(X)} \right\}\right] \right|\\
& \le \sqrt{\mathbb{E} \left[ \frac{(\Delta^{-\ell}
    f(X))^2}{\Delta^{-\ell} h(X)} \frac{K_p^{\ell}(X,
    X')}{p(X)p(X')} \Delta^{-\ell}h(X') \right]} \sqrt{\mathbb{E} \left[ \frac{(\Delta^{-\ell}
    g(X))^2}{\Delta^{-\ell} h(X)} \frac{K_p^{\ell}(X,
    X')}{p(X)p(X')} \Delta^{-\ell}h(X') \right]};
\end{align*}
using \eqref{eq:24} leads to the inequality. 

The only part of the claim that remains to be proved concerns the
saturation condition in the inequality. This follows from the
Cauchy-Schwarz inequality
which is an equality if and only if
$\Delta^{-\ell}f(x)/\Delta^{-\ell}h(x) \propto
\Delta^{-\ell}g(x')/\Delta^{-\ell}h(x') $ is constant throughout
$\mathcal{S}(p)$.  This is only possible under the stated condition.
\end{proof}

\begin{rmk}
{Theorem \ref{thm:klaassenbounds} can be refined using the exact
  expression for the remainder in the Cauchy-Schwarz inequality, given
  by the  Lagrange-type identity
  \begin{align*}
    & ( \mathbb{E} [ f(X_1, X_2) g(X_1, X_2)] )^2 \\
    & =\mathbb{E} [ f^2(X_1, X_2) ] \mathbb{E} [ g^2(X_1, X_2)]
- \frac12 \mathbb{E} [ (f(X_1, X_2) g(X_3, X_4) - f(X_3, X_4) g(X_1, X_2))^2 ]
  \end{align*}
with $X_1, X_2, X_3, X_4$ independent copies with density $p$ and $f, g \in L^2(p)$. }
Fix  $h \in L^1(p)$ a 
  decreasing function such that $  \| h \|_{{\mathcal{S}}(p), \infty}  <\infty$. For all $f, g$ which satisfy the assumptions of
  Lemma \ref{lem:stIBP1} we have 
  \begin{eqnarray*}
    \label{eq:KlaassenCov2}
(  \mathrm{Cov}[f(X), g(X)] )^2 &=& 
{\mathbb{E}\left[ (\Delta^{-\ell}f(X))^2
      \frac{-\mathcal{L}_p^{\ell}h(X)}{\Delta^{-\ell}h(X)} \right]}{\mathbb{E}\left[ (\Delta^{-\ell}g(X))^2
      \frac{-\mathcal{L}_p^{\ell}h(X)}{\Delta^{-\ell}h(X)} \right]} - \frac{1}{2}R(f,g,h) 
  \end{eqnarray*}
  with
  \begin{align*}
 R(f,g,h)  & =    \mathbb{E} \left[\left( \Delta^{-\ell}
             f(X_1) \Delta^{-\ell} g(X_4)                                               
            -  \Delta^{-\ell}
             f(X_3) \Delta^{-\ell} g(X_2)\frac{{
             \Delta^{-\ell}h(X_1)
             \Delta^{-\ell}h(X_4)}}{{\Delta^{-\ell} 
             h(X_2)\Delta^{-\ell} h(X_3)}}
    \right)^2 \right. \\ 
    &\quad \qquad \qquad  \left. \frac{{\Delta^{-\ell}h(X_2)
                                   \Delta^{-\ell}h(X_3)}}{{\Delta^{-\ell}
             h(X_1)\Delta^{-\ell} h(X_4)}}  \frac{K_p^{\ell}(X_1, 
             X_2)K_p^{\ell}(X_3, 
             X_4)}{p(X_1)p(X_2)p(X_3)p(X_4)} \right].
  \end{align*}
In particular when $h(x) = x$ the remainder term simplifies to 
\begin{align*}
 R(f,g,h)  & =    \mathbb{E} \left[\left( \Delta^{-\ell}
             f(X_1) \Delta^{-\ell} g(X_4)                                               
            -  \Delta^{-\ell}
             f(X_3) \Delta^{-\ell} g(X_2)
    \right)^2  \frac{K_p^{\ell}(X_1, 
             X_2)K_p^{\ell}(X_3, 
             X_4)}{p(X_1)p(X_2)p(X_3)p(X_4)} \right].
  \end{align*}
\end{rmk}

{Combining Proposition \ref{prop:cramerao} and Theorem
\ref{thm:klaassenbounds} (applied with $f=g$) we arrive at the
following result (applied to a smaller class of functions $h$) which,
as we shall argue below, share a similar flavour to the upper and lower
bounds from Theorem \ref{thm:klabou}.}
\begin{cor}[Klaassen bounds, revisited] \label{cor:klaaaaaassen}
  {For any decreasing function $h \in L^2(p)$ and all $g$ such
  that Lemma \ref{lem:stIBP1} applies (with $f = g$), we
  have
    \begin{equation}
    \label{eq:KlaassenVar}
  \frac{\mathbb{E}\left[ -\mathcal{L}_p^{\ell}h(X) (\Delta^{-\ell}g(X))\right]^{2}}{\mathrm{Var}
\big(h(X)\big)} \le  \mathrm{Var}[g(X)]   \le \mathbb{E}\left[ (\Delta^{-\ell}g(X))^2
      \frac{-\mathcal{L}_p^{\ell}h(X)}{\Delta^{-\ell}h(X)} \right].
  \end{equation}
  }
Equality in the upper bound holds if and only if there
exists constants $\alpha, \beta$ such that $g(x) = \alpha h(x) +
\beta$. 
\end{cor}
{
\begin{proof}
For the lower bound, we apply Proposition \ref{prop:cramerao}  with $f(x) = - \mathcal{L}_p^\ell h(x)$ so that 
$\mathcal{T}_p^\ell c(x) = h(x) -  - \mathbb{E}[h(X)].$ For the upper bound we use Theorem 
\ref{thm:klaassenbounds} with $f=g$.
\end{proof}
}


\begin{exm}[Pearson and Ord families] \label{def:pearsord}
  { { Tables \ref{tab:Ord1}, \ref{tab:Ord2}, and
      \ref{tab:Pearson} present the results for random variables whose
      distribution belongs to the Pearson and Ord families of
      distributions. A random variable $X \sim p$ belongs to the
      \emph{integrated Pearson family} if $X$ is absolutely continuous
      and there exist $\delta, \beta, \gamma \in \R$ not all equal to
      0 such that
      $\tau_p^{\ell}(x)\big(:=-\mathcal{L}_p^{\ell}(\mathrm{Id})\big)=\delta
      x^2 + \beta x + \gamma$ for all $x \in \mathcal{S}(p)$. {{Similarly,}}
      $X \sim p$ belongs to the cumulative Ord family if $X$ is
      discrete and there exist $\delta, \beta, \gamma \in \R$ not all
      equal to 0 such that
      $\tau_p^{\ell}(x)\big(:=-\mathcal{L}_p^{\ell}(\mathrm{Id})\big)=\delta
      x^2 + \beta x + \gamma$ for all $x \in \mathcal{S}(p)$.  The
      bounds for these distributions generalize the results e.g.\ from
      \cite{APP07}.}}
\end{exm}

{
  \begin{rmk}[About the connection with Klaassen's bounds] The bounds in Corollary \ref{cor:klaaaaaassen} and those
    from Theorem \ref{thm:klabou} are obviously of a similar flavour.
    Upon closer inspection, however, the connection is not
    transparent. In order to clarify this point, we follow \cite{K85}
    and restrict our attention to kernels of the form
    \begin{align*}
       \rho_{\zeta}^+(x, y) = \mathbb{I}[\zeta < y \le x] - \mathbb{I}[x < y
                     \le \zeta] \mbox{ and }
       \rho_{\zeta}^-(x, y) = \mathbb{I}[\zeta \le y < x] - \mathbb{I}[x \le y
                     < \zeta]
    \end{align*}
   for some $\zeta \in \R$. In our notations, these become
    $$\rho^{\ell}_{\zeta}(x,y) = \chi^{\ell}(\zeta, y)\chi^{-\ell}(y, x) -
    \chi^{\ell}(x, y)\chi^{-\ell}(y, \zeta)$$ for
    $\ell \in \left\{ -1, 0, 1 \right\}.$

    We first tackle the relation between the main arguments of the
    bounds, namely $G(x)$ and $g(x)$. Given a measurable function $g$,
    we mimic the statement of Theorem \ref{thm:klabou} and introduce
    the generalized primitive
    $G(x) = G_{\zeta}^{\ell}(x) := \int \rho_\zeta^{\ell}(x, y) g(y)
    \mu(\mathrm{d}y) + c$ with $c$ arbitrary, fixed w.l.o.g.\ to 0.
    Again in our notations, this becomes
\begin{equation*}
   G_{\zeta}^{\ell}(x) = \int_{\zeta+
   a_{\ell}}^{x-b_{\ell}} g(y) \mu(\mathrm{d}y) \mathbb{I}[\zeta<x] -  \int_{x+
   a_{\ell}}^{\zeta-b_{\ell}} g(y)
 \mu(\mathrm{d}y)\mathbb{I}[x<\zeta].
\end{equation*}
By construction, $\Delta^{-\ell}G_{\zeta}^{\ell}(x) = g(x)$ for all
$\zeta$ and all $\ell$, as expected. Nevertheless, in order for
$G_{\zeta}^\ell(x)$ to be well-defined, strong (joint) assumptions on
$g$ and $\zeta$ are required; for instance, if $g(x) = 1$ then $\zeta$
must be finite and $G_{\zeta, c}^{\ell}(x) = x - \zeta$ while if
$g(x)$ has $p$-mean 0 then the values $\zeta = \pm \infty$ are
allowed.

Next, we examine the connection between the lower bound \eqref{eq:klCR}
and the lower bound of \eqref{eq:KlaassenVar}.  Let $k \in L^2(p)$
have $p$-mean 0.  Then
\begin{align*}
  \mathbb{E}[k(X) \rho_{\zeta}^{\ell}(X, x)]  
  & =  \mathbb{E}[k(X) \chi^{-\ell}(x, X)] \chi^{\ell}(\zeta, x)
        - \mathbb{E}[k(X)\chi^{\ell}(X, x)]\chi^{-\ell}(x,\zeta)\\
  & = \mathbb{E}[k(X)]\chi^{\ell}(\zeta, x) 
      - \mathbb{E}[k(X) \chi^{\ell}(X,x)]\left(\chi^{\ell}(\zeta,x) + \chi^{-\ell}(x, \zeta) \right) \\
  & = - \mathbb{E}[k(X) \chi^{\ell}(X, x)] 
\end{align*}
so that
\begin{align*}
  K (x)  = \frac{1}{p(x)}\int \rho_{\zeta}^{\ell}(z, x) k(z) p(z) \mu(dz) = -\mathcal{L}_p^{\ell}k(x)
\end{align*}
and thus \eqref{eq:klCR} follows from  the  lower bound of
\eqref{eq:KlaassenVar}.

Finally, we consider the upper bounds \eqref{eq:klaass} and \eqref{eq:KlaassenVar}. Let $H(x) = H_{\zeta}^{\ell}(x)$
be a generalized primitive of some nonnegative function $h$. The same
manipulations as above lead to
\begin{align*}
  \frac{1}{p(x)} \int \rho_{\zeta}^{\ell}(z, x) H(z) p(z)
  \mu(\mathrm{d}z)   
  & = -\mathcal{L}_p^{\ell} h(x) - \frac{\mathbb{E}[H(X)]}{p(x)}
(P(x-a_{\ell})-\chi^{\ell}(\zeta, x)).
\end{align*}
If, following \cite{K85}, we choose $\zeta$ in such a way that
$\mathbb{E}[H(X)] = 0$ (this is equivalent to requiring
$ \int_{\zeta+ a_{\ell}}^b h(y) (1-P(y+b_{\ell})) \mu(\mathrm{d}y) =
\int_{a}^{\zeta- b_{\ell}}h(y) (P(y - a_{\ell})) \mu(\mathrm{d}y)$)
then we see that the upper bound in \eqref{eq:KlaassenVar} is 
equivalent to \eqref{eq:klaass}.

Of course there is some gain in generality at allowing for a general
kernel $\rho$ as in Theorem~\ref{thm:klabou}, though this comes at the
expense of readability: given a positive function $h$, understanding
the form of function $H$ is actually non trivial and our result
illuminates Klaassen's discovery by providing the connection with
Stein characterizations.
\end{rmk}
}


\section{About  the weights}
\label{sec:discussion}
The freedom of choice in the test functions $h$ appearing in {the}
bounds
{invite a} study {of} the impact of the choice of $h$ on the
validity and quality of the resulting inequalities.

\subsection{Score function and the Brascamp-Lieb inequality}
\label{sec:score-funct-brasc}

The form of the lower bound in Proposition \ref{prop:cramerao} encourages the
choice $f(x)= 1$.  This is only permitted if {the constant function} 
$1 \in \mathcal{F}_{\ell}^{(1)}(p)$ and
$\mathbb{E}\left[ \left(\mathcal{T}_p^{\ell}1(X)\right)^2\right]<\infty$; these are two strong
assumptions which exclude some natural targets such as e.g.\ the
exponential or beta distributions. If this choice is permitted, then
we reap the lower bound
\begin{align*}
  \mathrm{Var}[g(X)] \ge \frac{\mathbb{E}[\Delta^{-\ell}g(X)]}{I^{\ell}(p)}
\end{align*}
with
$I^{\ell}(p) = \left[ \left(\mathcal{T}_p^{\ell}1(X)\right)^2\right] 
$.

The function
$\mathcal{T}_p^{\ell}1(x) = \frac{\Delta^\ell (p(x))}{p(x)}$ is some form
of generalized score function and $I^{\ell}(p)$ a generalized Fisher
information. Indeed, if $\ell=0$ and $X\sim p$ is absolutely
continuous, then $\mathcal{T}_p^{\ell}1(x) = (\log p(x))'$ is exactly
the (location) score function of $p$ and $I^{(0)}(p)$ is none other
than the (location) Fisher information of $p$. More generally we note
that if $1 \in \mathcal{F}_{\ell}^{(1)}(p)$ then
$\mathbb{E}[\mathcal{T}_p^{\ell}1(X)]=0$ and, by Lemma \ref{lma:skadj}, it
satisfies
\begin{align*}
  \mathbb{E}[\mathcal{T}_p^{\ell}1(X) g(X)] = - \mathbb{E}[\Delta^{-\ell}g(X)]
\end{align*}
for all appropriate $g$; this further reinforces the analogy.

The corresponding upper bound from \eqref{eq:KlaassenCov} is obtained
for $h(x) = \mathcal{T}_p^{\ell}1(x)$ in
\eqref{eq:KlaassenVar}. Suppose that
$p(b^- + a_{\ell}) = p(a^+ - b_{\ell}) = 0$.  By construction,
$\mathcal{L}_p^{\ell}h(x) = \mathbb{I}_{\mathcal{S}(p)}(x)$. If we can
further suppose that $\mathcal{T}_p^{\ell}1(x)$ is a decreasing
function then
\begin{align*}
  \left| \mathrm{Cov}[f(X), g(X)] \right| \le
  \sqrt{\mathbb{E}\left[ 
      \frac{(\Delta^{-\ell}f(X))^2}{{-}\Delta^{-\ell}\mathcal{T}_p^{\ell}1(X)}
  \right]}\sqrt{\mathbb{E}\left[ 
  \frac{(\Delta^{-\ell}g(X))^2 }{{-}\Delta^{-\ell}\mathcal{T}_p^{\ell}1(X)} \right]}.
\end{align*}
Taking $g = f$ we deduce the following result whose continuous version
(i.e.\ the case $\ell = 0$) dates back to \cite{BrLi76}.

\begin{cor}[Brascamp-Lieb inequality]\label{cor:brascamplieb}
  Under the same conditions as Proposition \ref{prop:ottomenz1} we
  have
  \begin{equation}
    \label{eq:97}
      \frac{\mathbb{E}\left[(\Delta^{-\ell}g(X))\right]^{2}}{\mathbb{E}\left[\left(\mathcal{T}_p^{\ell}1(X)\right)^2\right]}
     \le     \mathrm{Var}[g(X)]   \le \mathbb{E}\left[
      \frac{(\Delta^{-\ell}g(X))^2}{{-}\Delta^{-\ell}\mathcal{T}_p^{\ell}1(X)}
    \right] 
  \end{equation}
  for all $g$ such that $\mathcal{T}_p^{\ell}1, g$ satisfy together
  the assumptions of Lemma \ref{lem:stIBP1}.
\end{cor}


We conclude with a generalized version of the elegant 
 inequality
due to \cite[Lemma 2.11]{menz2013uniform}, in the form stated in
\cite[Equation (1.5)]{carlen2013asymmetric}.

\begin{cor}[Asymmetric Brascamp-Lieb inequality]\label{cor:asbraslieb}
  Under the same conditions as above, if
  ${-}\Delta^{-\ell} \mathcal{T}_p^{\ell}1 \in L^1(\mu)$ then
  \begin{equation}
    \label{eq:98}
    \left| \mathrm{Cov}[f(X), g(X)] \right| \le \sup_x \left|
      \frac{\Delta^{-\ell}f(x)}{\Delta^{-\ell} \mathcal{T}_p^{\ell}1(x)} \right|
    \mathbb{E} \left[ \left| \Delta^{-\ell} g(X) \right| \right]
  \end{equation}
for all  $f, g$  in $L^2(p)$.
\end{cor}

\begin{proof}
  Under the stated assumptions, we may apply \eqref{eq:14} to get,
  after some notational reshuffling,
  \begin{align*}
|    \mathrm{Cov}[f(X), g(X)] | 
& \le \mathbb{E} \left[ \frac{\left| \Delta^{-\ell}f(X) \right|}
     {{-}\Delta^{-\ell} \mathcal{T}_p^{\ell}1(X)}
     |\Delta^{-\ell} g(X')| \left({-}\Delta^{-\ell}\mathcal{T}_p^{\ell}1(X)\right)
     \frac{K_p^{\ell} (X, X')}{p(X) p(X')}   \right] \\
    & 
    \le \sup_x \left|\frac{\Delta^{-\ell}f(x)}{\Delta^{-\ell}
      \mathcal{T}_p^{\ell}1(x)}\right|
    \mathbb{E} \left[ |\Delta^{-\ell} g(X')|
      \frac{K_p^{\ell} (X, X')}{p(X) p(X')}
     \left({-} \Delta^{-\ell}\mathcal{T}_p^{\ell}1(X) \right) \right]\\
&     
  = \sup_x \left|\frac{ \Delta^{-\ell}f(x)}
  				{\Delta^{-\ell}\mathcal{T}_p^{\ell}1(x)}\right|
    \mathbb{E} \left[ \left|\Delta^{-\ell} g(X')\right|    \right]
  \end{align*}
  where the last line follows by conditioning on $X'$ and applying
  Proposition \ref{prop:ottomenz1}. 
\end{proof}
\subsection{Stein kernel and Cacoullos' bound}
It is natural to consider test function $h = -\mathrm{Id}$  in Theorem
\ref{thm:klaassenbounds}. Since $\Delta^{-\ell}h(x) = 1$, we obtain 
\begin{align*}
  \left|   \mathrm{Cov}[f(X), g(X)]  \right| \le
  \sqrt{\mathbb{E}\left[ \tau_p^{\ell}(X) (\Delta^{-\ell}f(X))^2
       \right]}\sqrt{\mathbb{E}\left[ \tau_p^{\ell}(X) (\Delta^{-\ell}g(X))^2
      \right]}
\end{align*}
In particular if $g = f$ then
\begin{align*}
  \mathrm{Var}[g(X)] \le \mathbb{E} \left[ \tau_p^{\ell}(X)
  (\Delta^{-\ell}g(X))^2 \right] 
\end{align*}
in which one recognizes the upper bounds from \cite{C82} and also,
when $\ell = 0$, \cite{saumard2018weighted}. The corresponding lower
bound in \eqref{eq:13} is obtained for $f(x) = \tau_p^{\ell}(x)$ for
which $\mathcal{T}_p^{\ell} f(x) = x \mathbb{I}_{\mathcal{S}(p)}(x)$,
and the {overall} bound becomes
{
 \begin{equation}
    \label{eq:KlaassenVartau}
  \frac{\mathbb{E}\left[ \tau_p^{\ell}(X) (\Delta^{-\ell}g(X))\right]^{2}}{\mathrm{Var} [X] } \le  \mathrm{Var}[g(X)]   \le \mathbb{E}\left[ (\Delta^{-\ell}g(X))^2
      {\tau_p^{\ell}(X)} \right]. 
  \end{equation}}

\begin{exm} In our examples \eqref{eq:KlaassenVartau} gives the following covariance identities.
\begin{itemize}
\item\label{ex:binom3}  Binomial distribution:  
  Let $X\sim \mathrm{Bin}(n, \theta)$ as in Example \ref{ex:binom1}.  From Example \ref{ex:binom2} we
  obtain the upper and lower bounds 
\begin{align*}
&   \frac{(1-\theta) }{n\theta} \mathbb{E} \left[ X \Delta^-g(X) \right]^2 \le
  \mathrm{Var}[g(X)] \le (1-\theta)\mathbb{E} \left[ X(\Delta^-g(X))^2
  \right]; \\
& \frac{\theta }{n(1-\theta)} \mathbb{E} \left[(n-X) \Delta^+g(X) \right]^2 \le
  \mathrm{Var}[g(X)] 
 \le  \theta \mathbb{E} \left[(n-X)(\Delta^+g(X))^2
  \right].  
\end{align*}
\item\label{ex:beta3} Beta distribution:  From Example \ref{ex:beta2}, for the $\mathrm{Beta}(\alpha, \beta)$-distribution with variance $\frac{\alpha \beta}{(\alpha + \beta)^2 (\alpha + \beta + 1)}$,
\begin{align*}
 &  \frac{(\alpha+\beta+1)}{\alpha\beta} \mathbb{E} \left[ X(1-X) g'(X) \right]^2 \le
  \mathrm{Var}[g(X)] \le \frac{1}{\alpha+\beta} \mathbb{E} \left[ X(1-X)(g'(X))^2
  \right].    
\end{align*} 
\end{itemize}
\end{exm}

The particular case of other Pearson/Ord distributions  is detailed in
Tables \ref{tab:Ord1}, \ref{tab:Ord2} and \ref{tab:Pearson}.  
The tables include the Binomial distribution and the Beta distribution for easy reference. The Stein operators which are given are those from Example \ref{ex:genera}.

\subsection{Eigenfunctions of the adjoint Stein operator}
\label{sec:eigenf-adjo-stein}

{A final} interesting choice is $h$ in Theorem \ref{thm:klaassenbounds} such
that the corresponding weight
{$\frac{-\mathcal{L}_p^{\ell}h(x)}{\Delta^{-\ell}h(x)}$} is constant,
i.e.\ any function $h$ such that there exists $\lambda \in \R$ for
which
\begin{align*}
   \frac{-\mathcal{L}_p^{\ell}h(x)}{\Delta^{-\ell}h(x)} = \lambda
  \mbox{ for all } x \in \mathcal{S}(p). 
\end{align*}
By construction, such  functions are solution to the eigenfunction
problem
\begin{align*}
h(x) =  {- \lambda} \mathcal{T}_p^{\ell}(\Delta^{-\ell}h)(x)   \mbox{ for all } x
  \in \mathcal{S}(p)
\end{align*}
where operator $\mathcal{R}_p^{\ell}h :=
\mathcal{T}_p^{\ell}(\Delta^{-\ell}h)$ is self-adjoint in the sense of
that
\begin{align*}
  \mathbb{E}[(\mathcal{R}_p^{\ell}f(X))  g(X)] =   \mathbb{E}[f(X) ( \mathcal{R}_p^{\ell}g(X))]
\end{align*}
for all $f, g$ such that Lemmas \ref{lma:skadj} and \ref{lem:stIBP1}
apply.

\section*{Acknowledgements}

The research of YS was partially supported by the Fonds de la
Recherche Scientifique -- FNRS under Grant no F.4539.16.  ME
acknowledges partial funding via a Welcome Grant of the Universit\'e
de Li\`ege and via the Interuniversity Attraction Pole StUDyS
(IAP/P7/06).  YS also thanks Lihu Xu for organizing the ``Workshop on
Stein's method and related topics'' at University of Macau in December
2018, and where a preliminary version of this contribution was first
presented. GR and YS thank Emilie Clette for fruitful discussions on a
preliminary version of this work. YS thanks C\'eline Esser  for many fruitful
discussions. We also thank Benjamin Arras for
several pointers to relevant literature, as well as corrections on the
first draft of the paper. 


\begin{sidewaystable}
\begin{tabular}{llll}\hline
  \textbf{name} & \textbf{p.m.f. $p(x)$} & \textbf{Stein kernel $\tau^{\ell}(x)$}  \\ 
  \textbf{parameter} & \textbf{support} & \textbf{Cum. Ord relation} &  \\ 
  \hline
  Poisson $(\lambda)$ & 
                        $e^{-\lambda}\lambda^x/x!$ & 
                                                              $\tau^-(x)=\lambda$ 
  \\
  $\lambda>0$ & $x=0,1,\ldots$ & 
                                 $\tau^+(x)=x$
  \\
                & & $(\delta,\beta,\gamma)=(0,0,\lambda)$ 
  \\ \hdashline[0.5pt/5pt]
                & Stein operators & \multicolumn{2}{l}{$ \mathcal{A}^+_{\mathrm{Poi}(\lambda)}g(x) =  (x-\lambda) g(x) -x \Delta^-g(x) $ }\\
                && \multicolumn{2}{l}{$ \mathcal{A}^-_{\mathrm{Poi}( \lambda)}g(x) = (x-\lambda)g(x) - \lambda \Delta^+g(x)$} \\
                && \multicolumn{2}{l}{$ \mathcal{A}_{\mathrm{Poi}( \lambda)}g(x) = x g(x) - \lambda g(x+1)$} \\ \hdashline[0.5pt/5pt]
                & Variance bounds& \multicolumn{2}{l}{$ \lambda \mathbb{E} \left[ \Delta^+g(X) \right]^2 \le \mathrm{Var}[g(X)] \le \lambda \mathbb{E} \left[  (\Delta^+g(X))^2 \right]$} \\
                & &  \multicolumn{2}{l}{$ \lambda^{-1}\mathbb{E} \left[ X \Delta^-g(X) \right]^2 \le \mathrm{Var}[g(X)] \le \mathbb{E} \left[ X(\Delta^-g(X))^2  \right]$}\\ \hline
Binomial $(n,\theta)$ & 
$\binom{n}{x}\theta^x(1-\theta)^{n-x}$ & 
$\tau^-(x)=\theta(n-x)$ \\
$0<\theta<1$ & $x=0,1,\ldots,n$ & 
$\tau^+(x)=(1-\theta)x$ \\
$n=1,2,\ldots$ & & $(\delta,\beta,\gamma)=(0,-\theta,n\theta)$ \\ \hdashline[0.5pt/5pt]
& Stein operators &
 \multicolumn{2}{l}{$ \mathcal{A}^+_{\mathrm{bin}(n,\theta)}g(x) =  (x-n\theta) g(x) -(1-\theta)x \Delta^-g(x) $ }\\
&& \multicolumn{2}{l}{$ \mathcal{A}^-_{\mathrm{bin}(n,\theta)}g(x) = (x-n\theta)g(x) - \theta(n-x) \Delta^+g(x)$} \\
&& \multicolumn{2}{l}{$ \mathcal{A}_{\mathrm{bin}(n,\theta)}g(x) = xg(x) + \frac{\theta}{1-\theta}(n-x)g(x+1)$} \\ \hdashline[0.5pt/5pt]
& Variance bounds& \multicolumn{2}{l}{$\frac{\theta}{n(1-\theta)} \mathbb{E} \left[(n-X)
    \Delta^+g(X) \right]^2 \le \mathrm{Var}[g(X)] \le \theta \mathbb{E}
    \left[ (n-X) (\Delta^+g(X))^2 \right]$} 
 \\
& & \multicolumn{2}{l}{$ \frac{1-\theta}{n\theta}\mathbb{E} \left[ X \Delta^-g(X) \right]^2 \le
 \mathrm{Var}[g(X)] \le (1-\theta)\mathbb{E} \left[ X(\Delta^-g(X))^2  \right]$}\\  \hline      
Negative Binomial $(r,p)$ & 
$\binom{r+x-1}{x}p^r(1-p)^{x}$ & 
$\tau^-(x)=\frac{1-p}{p}(r+x)$ \\
$0<p<1$ & $x=0,1,\ldots$ & 
$\tau^+(x)=\frac{1}{p}x$ 
\\
$r>0$ & & 
$(\delta,\beta,\gamma)=(0,\frac{1-p}{p},r\frac{1-p}{p})$ 
\\ 
\hdashline[0.5pt/5pt]
& Stein operators & \multicolumn{2}{l}{$ \mathcal{A}^+_{\mathrm{NB}(r,p)}g(x) =  \left(x-\frac{1-p}{p}r\right) g(x) -\frac{x}{p} \Delta^-g(x) $ }\\
&& \multicolumn{2}{l}{$ \mathcal{A}^-_{\mathrm{NB}(r,p)}g(x) = \left(x-\frac{1-p}{p}r\right) g(x) -\frac{1-p}{p}(r+x) \Delta^+g(x)$} \\
&& \multicolumn{2}{l}{$ \mathcal{A}_{\mathrm{NB}(r,p)}g(x) = xg(x) - (1-p)(r+x)g(x+1)$} \\ \hdashline[0.5pt/5pt]
& Variance bounds&   \multicolumn{2}{l}{$ \frac{1-p}{r} \mathbb{E}[(X+r)\Delta^+g(X)]^2  \le \mathrm{Var}[g(X)] \le  \frac{1-p}{p}\mathbb{E}[(X+r) (\Delta^+g(X))^2] $}\\
& &\multicolumn{2}{l}{$  \frac{1}{r(1-p)}\mathbb{E}[X\Delta^-g(X)]^2 \le
 \mathrm{Var}[g(X)] \le \frac{1}{p}\mathbb{E}[X (\Delta^-g(X))^2] $} \\  \hline     
\end{tabular} 
\caption{\label{tab:Ord1}Specific forms some discrete distributions
  from the cumulative Ord family. This table is a completed version of
  Table 1 of \cite{APP07}.}
\end{sidewaystable}

\begin{sidewaystable}
\begin{tabular}{llll}\hline
  \textbf{name} & \textbf{p.m.f. $p(x)$} & \textbf{Stein kernel $\tau^{\ell}(x)$}  \\ 
  \textbf{parameter} & \textbf{support} & \textbf{Cum. Ord relation}  \\ 
  \hline
  Hypergeometric & 
 $\frac{\binom{K}{x}\binom{N-K}{n-x}}{\binom{N}{n}}$ &  
$\tau^-(x)=\frac{1}{N}(K-x)(n-x)$ \\
  $(n,K,N)$   & $x=0,1,\ldots,\min\{K,n\}$ & 
  $\tau^+(x)=\frac{1}{N}x(x+N-K-n)$  \\
  $1 \leq K \leq N$& &
  $(\delta,\beta,\gamma)=\left(\frac{1}{N},\frac{-(n+K)}{N},\frac{nK}{N}\right)$  \\
  $n=1,2,\ldots,N$ &&&
  \\ \hdashline[0.5pt/5pt]
 & Stein operators & \multicolumn{2}{l}{$ \mathcal{A}^+_{\mathrm{H}(n,K,N)}g(x) = \left(x-\frac{nK}{N}\right)g(x) - \frac{1}{N}x (N-K-n+x)\Delta^-g(x)  $ }\\
                && \multicolumn{2}{l}{$ \mathcal{A}^-_{\mathrm{H}(n,K,N)}g(x) = \left(x-\frac{nK}{N}\right)g(x) - \frac{1}{N}(K-x) (n-x)\Delta^+g(x)$} \\
                && \multicolumn{2}{l}{$ \mathcal{A}_{\mathrm{H}(n,K,N)}g(x) =xg(x) +\frac{1}{N+K+n}x^2g(x) - \frac{1}{N+K+n}(K-x)(n-x)g(x+1) $} \\ 
  \hdashline[0.5pt/5pt]
                & Variance bounds& \multicolumn{2}{l}{$ \frac{N-1}{nK(N-K)(N-n)} \mathbb{E}[(K-X)(n-X)\Delta^+g(X)]^2 \le \mathrm{Var}[g(X)] \le  \frac{1}{N} \mathbb{E}[(K-X)(n-X)(\Delta^+g(X))^2 ]$} \\
                & & \multicolumn{2}{l}{$ \frac{N-1}{nK(N-K)(N-n)} \mathbb{E}[X(N-K-n+X)\Delta^-g(X)]^2 \le
                    \mathrm{Var}[g(X)] \le  \frac{1}{N} \mathbb{E}[X(N-K-n+X)(\Delta^-g(X))^2 ]  $} \\ \hline
Negative Hyper- & 
$\frac{\binom{x+r-1}{x}\binom{N-r-x}{K-x}}{\binom{N}{K}}$ & 
$\tau^-(x)=\frac{1}{N-K+1}(K-x)(r+x)$ \\
geometric $(N,K,r)$& $x=0,1,\ldots,K$ & 
$\tau^+(x)=\frac{1}{N-K+1}x(N-r+1-x)$ \\
$0\leq K\leq N$ & & $(\delta,\beta,\gamma)=\left(\frac{-1}{N-K+1},\frac{K-r}{N-K+1},\frac{rK}{N-K+1}\right)$ &\\ 
 \hdashline[0.5pt/5pt]
& Stein operators & \multicolumn{2}{l}{$ \mathcal{A}^+_{\mathrm{NH}(N,K,r)}g(x) =  \left(x-\frac{rK}{N-K+1}\right)g(x) - \frac{x(N+1-r-x)}{N-K+1}\Delta^-g(x)  $ }\\
&& \multicolumn{2}{l}{$ \mathcal{A}^-_{\mathrm{NH}(N,K,r)}g(x) = \left(x-\frac{rK}{N-K+1}\right)g(x) - \frac{(K-x)(r+x)}{N-K+1}\Delta^+g(x)$} \\
&& \multicolumn{2}{l}{$ \mathcal{A}_{\mathrm{NH}(N,K,r)}g(x) = xg(x) - \frac{1}{N-r+1}x^2g(x) - \frac{1}{N-r+1}(K-x)(r+x) g(x+1)$} \\[5pt] \hdashline[0.5pt/5pt]
& Variance bounds& 
\multicolumn{2}{l}{$   \frac{N-K+2}{r(N+1)K(N-K-r+1)} \mathbb{E}[(K-X)(r+X)\Delta^+g(X)]^2 \le
 \mathrm{Var}[g(X)]  $} \\[5pt]
 & & \multicolumn{2}{l}{$\mathrm{Var}[g(X)] \le \frac{1}{N-K+1} \mathbb{E}[(K-X)(r+X)(\Delta^+g(X))^2]$ } \\
& &\multicolumn{2}{l}{$ \frac{N-K+2}{r(N+1)K(N-K-r+1)} \mathbb{E}[X(N+1-r-X)\Delta^-g(X)]^2 \le
 \mathrm{Var}[g(X)]$} \\[5pt]
 & & \multicolumn{2}{l}{$\mathrm{Var}[g(X)] \le \frac{1}{N-K+1} \mathbb{E}[X(N+1-r-X)(\Delta^-g(X))^2] $}\\[5pt]
 \hline      
\end{tabular} 
\caption{\label{tab:Ord2}Specific form for some discrete distributions
  from the cumulative Ord family (second part).}
\end{sidewaystable}

\begin{sidewaystable}
\begin{tabular}{llll}\hline
\textbf{name} & \textbf{p.m.f. $p(x)$} & \textbf{$\tau(x)$}  \\ 
\textbf{parameter} & \textbf{support} & \textbf{Pearson relation} &  \\ 
\hline
Normal$(\mu,\sigma^2)$ & $\frac{1}{\sqrt{2\pi\sigma^2}}\exp\left(\frac{(x-\mu)^2}{2\sigma^2}\right)$ & $\tau(x)= \sigma^2$  \\
$\mu\in\R$, $\sigma^2>0$ & $x\in \R$ & $(\delta,\beta,\gamma)=(0,0,\sigma^2 )$  &  \\[5pt] \hdashline[0.5pt/5pt]
& Stein operators & \multicolumn{2}{l}{$ \mathcal{A}_{\mathrm{N(\mu,\sigma^2)}}g(x) = (x-\mu)g(x)-\sigma^2 g'(x)$ }
 \\[5pt] \hdashline[0.5pt/5pt]
& Variance bounds& \multicolumn{2}{l}{$ \sigma^2\mathbb{E}[g'(X)]^2  \le
  \mathrm{Var}[f(X)] \le \sigma^2\mathbb{E}[g'(X)^2] $} \\[5pt]
 \hline
Beta $(\alpha,\beta)$ & 
$x^{\alpha-1}(1-x)^{\beta-1}/B(\alpha,\beta)$ & 
$\tau(x)= \frac{x(1-x)}{\alpha+\beta} $ 
 \\
$\alpha>0$, $\beta>0$ & $x\in (0,1)$ & $(\delta,\beta,\gamma)=( \frac{-1}{\alpha+\beta},\frac{1}{\alpha+\beta},0 )$  & \\[5pt] \hdashline[0.5pt/5pt]
& Stein operators & \multicolumn{2}{l}{$\mathcal{A}_{\mathrm{Beta(\alpha,\beta)}}g(x) = \left(x-\frac{\alpha}{\alpha+\beta}\right)g(x)-\frac{x(1-x)}{\alpha+\beta} g'(x)$ }\\[5pt] \hdashline[0.5pt/5pt]
& Variance bounds& \multicolumn{2}{l}{$\frac{(\alpha+\beta+1)}{\alpha\beta} \mathbb{E} \left[ X(1-X) g'(X) \right]^2 \le
  \mathrm{Var}[g(X)] \le \frac{1}{\alpha+\beta} \mathbb{E} \left[ X(1-X)(g'(X))^2
  \right]   $} \\[5pt]
 \hline
Gamma$(\mu,\sigma^2)$ &
                        $x^{\alpha-1}\beta^{-\alpha}e^{-x/\beta}/\Gamma(\alpha)$ & $\tau(x)=\beta x $  \\
$\alpha>0$, $\beta>0$ & $x\in (0,\infty) \ (\alpha<1)$ & $(\delta,\beta,\gamma)=(0,\beta,0 )$  & \\
& $x\in [0,\infty)\ (\alpha\geq 1)$&  & \\[5pt] \hdashline[0.5pt/5pt]
& Stein operators & \multicolumn{2}{l}{$ \mathcal{A}_{\mathrm{Gamma(\alpha,\beta)}}g(x) = (x-\alpha \beta)g(x)-\beta x g'(x)$ }\\[5pt] \hdashline[0.5pt/5pt]
& Variance bounds& \multicolumn{2}{l}{$ \frac{1}{\alpha}\mathbb{E}[X g'(X)]^2  \le
  \mathrm{Var}[g(X)] \le \beta \mathbb{E}[X g'(X)^2] $} \\[5pt]
 \hline  
  {Student $(\nu)$} & $\nu^{-1/2} B(\nu/2,
  1/2)^{-1}( \nu/(\nu+x^2))^{(1+\nu)/2}$ & 
                                  {$\tau(x)=  \frac{x^2+\nu}{\nu-1} $ for $\nu>1$}  \\
  $\nu>0$ & $x\in \R$ & $(\delta,\beta,\gamma)=\left(\frac{1}{\nu-1}, 0, \frac{\nu}{\nu-1} \right)$  &  \\[5pt] \hdashline[0.5pt/5pt]
              & Stein operators & \multicolumn{2}{l}{$ \mathcal{A}_{\mathrm{t(\nu)}}g(x) =x g(x) - \frac{x^2 + \nu}{\nu-1} g'(x)  $ }\\[5pt] \hdashline[0.5pt/5pt]
              & Variance bounds ($\nu>2$)& \multicolumn{2}{l}{$ \frac{{(\nu-2)}}{{\nu} (\nu-1)^2} \mathbb{E}[(X^2 + \nu) g'(X)]^2 \le
                                 \mathrm{Var}[g(X)] \le \frac{1}{\nu-1} \mathbb{E}[(X^2 + \nu) ( g'(X)^2)]  $} \\[5pt]
 \hline  
F distribution  $(d_1,d_2)$ & $\frac{\left(\frac{d_1}{d_2}\right)^{d_1/2}x^{\frac{d_1}{2}-1} \left(1+ \frac{d_1}{d_2}x \right)^{-\frac{d_1+d_2}{2}}}{\mbox{Beta}(d_1/2,d_2/2)} $ & 
$\tau(x)= \frac{2x(d_1x+d_2)}{d_1(d_2-2)}$ for $d_2>2$  \\
$d_1>0$, $d_2>0$ & $x\in (0,\infty)$ & $(\delta,\beta,\gamma)=\left(\frac{2d_1}{d_1(d_2-2)}, \frac{2d_2}{d_1(d_2-2)},0 \right)$  &  \\[5pt] \hdashline[0.5pt/5pt]
& Stein operators & \multicolumn{2}{l}{$ \mathcal{A}_{\mathrm{F(d_1,d_2)}}g(x) = \left(x-\frac{d_2}{d_2-1}\right)g(x) - \frac{2x(d_2+d_1x}{d1(d_2-2)}g'(x)$ }\\[5pt]
 \hdashline[0.5pt/5pt]
& Variance bounds ($d_2>4$) & \multicolumn{2}{l}{$  {\frac{2(d_2-4)}{d_1d_2^2(d_1+d_2-2)}} \mathbb{E}[X (d_2+d_1X)g'(X)]^2 \le
 \mathrm{Var}[g(X)] \le {\frac{2}{d_1(d_2-2)}}\mathbb{E}[X(d_2+d_1X)g'(X)^2]  $} \\[5pt]
 \hline  
\end{tabular} 
\caption{\label{tab:Pearson}Specific form for some continuous
  distributions. If the distribution belongs to the Pearson family
  (examples from \cite{afendras2011matrix}), the coefficient
  $(\delta,\beta,\gamma)$ are given. This table is an adapted version
  of Table 1 of \cite{APP07}.}
\end{sidewaystable}

\bibliographystyle{abbrv} 

 \appendix
 \section{Proofs from Section \ref{sec:suff-cond}}
\label{sec:some-non-essential}

\begin{proof}[Proof of Proposition \ref{prop:suff1}]
 { In order for \eqref{eq:3} to hold we need (i)
  $f(\cdot) g(\cdot - \ell) \in \mathcal{F}^{(1)}_{\ell}(p)$ and (ii)
  $f(\cdot) \Delta^{-\ell} g(\cdot) \in L^1(p)$. Condition (ii) is
  satisfied under \eqref{eq:11}.  By definition of
  $\mathcal{F}^{(1)}_{\ell}(p)$, condition (i) holds if the following
  three conditions apply: (iA)
  $f(\cdot) g(\cdot - \ell) \in \mathrm{dom}(p, \Delta^{\ell})$, (iB)
  $\Delta^{\ell} \big(p(\cdot)f(\cdot) g(\cdot-\ell)\big) \mathbb{I}[
  \mathcal{S}(p)] \in L^1(\mu)$ and (iC)
  $\mathbb{E} \left[ \mathcal{T}_p^{\ell}f(X) g(X-\ell) \right] =
  0$. The proof hinges on product formula \eqref{eq:productrule} which
  yields:
  \begin{equation*}
    \Delta^{\ell} \big(p(x)f(x) g(x-\ell)\big) =  \Delta^{\ell}
    \big(p(x)f(x)\big)  g(x) + p(x) f(x) \Delta^{-\ell} g(x)
  \end{equation*}
  for all $x \in \mathcal{S}(p)$.  In light of this, condition (iA) is
  implied by the requirement that
  $ f\in \mathrm{dom}(p, \Delta^{\ell})$ and
  $g \in \mathrm{dom}(\Delta^{-\ell})$. Similarly, because condition
  (iB) is equivalent to
  $\frac{\Delta^{\ell} \big(p(\cdot)f(\cdot) g(\cdot-\ell)\big)
  }{p(\cdot)} \in L^1(p)$, we see that it is guaranteed by
  \eqref{eq:11}. Finally, applying \eqref{eq:30}, we that (iC) follows
  from \eqref{eq:9}. Hence Condition (i) holds under the stated
  assumptions. 
}


\end{proof}

{\begin{proof}[Proof of Proposition \ref{prop:ibp2}] In order for \eqref{eq:21} to hold, it is necessary and
  sufficient that (i) $f \in L^1(p)$,
  $g \in \mathrm{dom}(\Delta^{-\ell})$, (ii)
  $\big(\mathcal{L}_p^{\ell} f(\cdot)\big) g(\cdot-\ell)$ and (iii)
  $\mathcal{L}_p^{\ell} f ( \Delta^{-\ell}g) \in L^1(p)$. Conditions
  (i) and (iii) are stated explicitly and all that remains is to check
  that (ii) is equivalent to the stated assumptions. As before, we
  recall that (ii) is equivalent to (iiA) $\big(\mathcal{L}_p^{\ell}
  f(\cdot)\big) g(\cdot-\ell) \in \mathrm{dom}(p, \Delta^{\ell})$;
  (iiB) $\Delta^{\ell} \bigg(p(\cdot)\big(\mathcal{L}_p^{\ell}
  f(\cdot)\big) g(\cdot-\ell)\bigg)/p(\cdot) \in L^1(p)$; (iiC)
  $\mathbb{E} \left[ \mathcal{T}_p^{\ell}\left(
      \big(\mathcal{L}_p^{\ell} f(\cdot)\big) g(\cdot-\ell) \right)
    (X) \right] = 0$. As in the proof of Proposition \ref{prop:suff1},
  the result hinges on the product rule \eqref{eq:productrule} which
  now reads
  \begin{align*}
    \Delta^{\ell} \left( p(x) \big(\mathcal{L}_p^{\ell} f(x)\big)
      g(x-\ell) \right) & = \bigg(\Delta^{\ell} \big(p(x)
    \mathcal{L}_p^{\ell} f(x)\big)\bigg) g(x) + p(x)
                          \mathcal{L}_p^{\ell} f(x) \Delta^{-\ell} g(x) \\
    &   =  (f(x) - \mathbb{E}[f(X)]) g(x) + p(x)
                          \mathcal{L}_p^{\ell} f(x) \Delta^{-\ell} g(x) 
  \end{align*}
  Hence condition (iiA) holds solely under the assumption that
  $g \in \mathrm{dom}(\Delta^{-\ell})$, condition (iiA) holds under
 \eqref{eq:91} and  \eqref{eq:18}. Finally, \eqref{eq:92} guarantees
 that (iiC) is satisfied. 
\end{proof}
}

{\begin{proof}[Proof of Proposition \ref{lma:suffcond1}] We {want
    to apply Proposition \ref{prop:ibp2}; hence we} check each
  condition in Proposition \ref{prop:ibp2} separately. By assumption,
  \eqref{eq:91} is satisfied and $g \in \mathrm{dom}(\Delta^{-\ell})$.  
  
  \begin{itemize}
  \item For Assumption \eqref{eq:18}: First suppose that $g$ is
    monotone increasing.  It is to show that
    $\mathcal{L}_p^{\ell} f ( \Delta^{-\ell} g )\in L^1(p)$.  {As
      $f \in L^1(p)$ is assumed,} we can use \eqref{eq:reprsntform}
    to get
   \begin{align*}
     \mathbb{E} \left[ \left| \mathcal{L}_p^{\ell}f(X) \right| \left|
     \Delta ^{-\ell}g(X) \right| \right] & =
\mathbb{E} \left[ \left| \mathcal{L}_p^{\ell}f(X) \right| 
                                           \Delta ^{-\ell}g(X)
                                           \right]\\
                                         & \le \mathbb{E}  \left[  \left|
                                           f(X_2) - f(X_1) \right|
                                                     \Phi_p^{\ell}(X_1, X,
                                                     X_2) \Delta^{-\ell}g(X)
                                                     \right]              \\                            
& \le \mathbb{E}  \left[  \left| f(X_2) - f(X_1) \right|
                                                    \mathbb{E}\bigg[ \Phi_p^{\ell}(X_1, X, X_2)
                                                                                                           \Delta^{-\ell}g(X)  \, | \, X_1, X_2\bigg]\right]  \\
     & = \mathbb{E}  \bigg[  \left| f(X_2) - f(X_1) \right|
       \left( g(X_2)-g(X_{1}) \right)
\mathbb{I}[X_1<X_2]       \bigg]
   \end{align*}
   where we used the first identity in \eqref{eq:28} in the last line.
   This last expression is necessarily finite because $f, g$ and $fg$
   are in $L^1(p)$.  The general conclusion follows from the fact that
   any function of bounded variation is the difference between two
   monotone functions; {the triangle inequality thus} yielding the
   claim.
 
 \item For Assumption \eqref{eq:92}: Since $f \in L^1(p)$, we can
   apply \eqref{eq:reprsntform} and the definition of $\Phi_p^{\ell}$
   to obtain
{\begin{align*}
-   p(x) \mathcal{L}_p^{\ell} f(x) =   \mathbb{E}[f(X)\chi^{-\ell}(x, X)] \mathbb{E} [ \chi^{\ell}(X,
    x)] - \mathbb{E} [f(X) \chi^{\ell}(X,
    x)] \mathbb{E}[\chi^{-\ell}(x, X)]. 
    \end{align*}
    }
   Then
  \begin{align*}
& \lim_{x \to a, x>a} \left| \bigg(\mathcal{L}_p^{\ell}
    f(x-b_{\ell})\bigg)  g(x-b_{\ell}-\ell)  p(x-b_{\ell})\right| \\
    &  \le    \lim_{x \to a, x>a} \bigg( | g(x-b_{\ell}-\ell) |
                      \mathbb{E} [|f(X)| \chi^{-\ell}( 
                      x-b_{\ell}, X)]   \mathbb{E} [\chi^{\ell}(X, 
                      x-b_{\ell})] \\
                    & \qquad +  |
                      g(x-b_{\ell}-\ell) | \mathbb{E}[|f(X)|\chi^{\ell}(X, x-b_{\ell})]   \mathbb{E} [\chi^{-\ell}(x-b_{\ell}, 
                      X)] \bigg)\\
   & \le   \lim_{x \to a, x>a}  | g(x-b_{\ell}-\ell) |
                      \mathbb{E} [|f(X)| \chi^{-\ell}( 
                      x-b_{\ell}, X)]   \mathbb{P} \left( X \le  
                      x-a_{\ell}- b_{\ell} \right) & =L_1 \\
                    & +  \lim_{x \to a, x>a} 
                      |g(x-b_{\ell}-\ell) |
                      \mathbb{E}[|f(X)|\chi^{\ell}(X, x-b_{\ell})]
                      \mathbb{P} \left( X \ge x \right) & =L_2
  \end{align*}
and 
    \begin{align*}
     &  \lim_{x \to b, x<b} \left| \bigg(\mathcal{L}_p^{\ell}
      f(x+a_{\ell}) \bigg) g(x+a_{\ell}-\ell)  p(x+a_{\ell}) \right|\\
       &  \le    \lim_{x \to b, x<b} \bigg( | g(x+a_{\ell}-\ell) |
                      \mathbb{E} [|f(X)| \chi^{-\ell}( 
                      x+a_{\ell}, X)]   \mathbb{E} [\chi^{\ell}(X, 
                      x+a_{\ell})] \\
                    & \qquad +  |
                      g(x+a_{\ell}-\ell) | \mathbb{E}[|f(X)|\chi^{\ell}(X, x+a_{\ell})]   \mathbb{E} [\chi^{-\ell}(x+a_{\ell}, 
                      X)] \bigg)\\
       &  \le    \lim_{x \to b, x<b} | g(x+a_{\ell}-\ell) |
                      \mathbb{E} [|f(X)| \chi^{-\ell}( 
                      x+a_{\ell}, X)]   \mathbb{P}\left(X \le  
                      x\right) & =L_3 \\
                    & +  \lim_{x \to b, x<b}   |
                      g(x+a_{\ell}-\ell) |
                      \mathbb{E}[|f(X)|\chi^{\ell}(X, x+a_{\ell})]
                      \mathbb{P} \left( X \ge x + a_{\ell} + b_{\ell}) \right)& =L_4.
    \end{align*} 
    Condition \ref{item:cond1g} guarantees that $L_1 = L_4 = 0$;
    condition \ref{item:cond2g} guarantees that $L_2 = L_3 = 0$. If,
    furthermore, $f$ is bounded then the sufficiency of
    \ref{item:cond1g} is immediate; if $f \in L^2(p)$ then it follows
    from the Cauchy-Schwarz inequality.
    \end{itemize}
    \end{proof} 
}

\end{document}